\documentclass[review]{siamart220329}

\usepackage{lipsum}
\nolinenumbers
\usepackage{amsfonts}
\usepackage{graphicx}
\usepackage{epstopdf}
\usepackage{algorithmic}
\usepackage{multirow}
\ifpdf%
  \DeclareGraphicsExtensions{.eps,.pdf,.png,.jpg}
\else
  \DeclareGraphicsExtensions{.eps}
\fi
\usepackage{amsopn}

\usepackage{booktabs}

\def\range{{\cal R}}

\newcommand{\TheTitle}{
Inner Product Free 
Krylov Methods for 
Large-Scale Inverse Problems
}

\author{Ariana~N.~Brown\footnote{Department of Mathematics, Emory University, Atlanta, GA, 30322}
\and Julianne Chung\footnotemark[1] \and James~G.~Nagy\footnotemark[1] \and Malena Sabat\'e Landman\footnotemark[1]}
\usepackage{amsopn}

\title{{\TheTitle}}
\headers{Inner Product Free Methods for Inverse Problems}{Brown, Chung, Nagy, Sabaté Landman}

\begin{document}

\maketitle

\vspace{1cm}

\begin{abstract}
  In this study, we introduce two new Krylov subspace methods for solving rectangular large-scale linear inverse problems. The first approach is a modification of the Hessenberg iterative algorithm that is based off an LU factorization and is therefore referred to as the least squares LU (LSLU) method. The second approach incorporates Tikhonov regularization in an efficient manner; we call this the Hybrid LSLU method. Both methods are inner-product free, making them advantageous for high performance computing and mixed precision arithmetic. Theoretical findings and numerical results show that Hybrid LSLU can be effective in solving large-scale inverse problems and has comparable performance with existing iterative projection methods.
\end{abstract}



\section{Introduction}\label{sec:intro}
Inverse problems are present in many different applications including medical and geophysical imaging, electromagnetic scattering, machine learning, and image deblurring \cite{Chung2011Inverse,Hansen2010,Vogel2002,Zhdanov2002}. We consider a large linear inverse problem of the form: \begin{equation} \label{eq:ip} b = Ax_{\rm true} + e, \end{equation} where $A \in \mathbb{R}^{m \times n}$ models the forward problem, $x_{\rm true} \in \mathbb{R}^{n}$ is the unknown solution we want to approximate, $b \in \mathbb{R}^{m}$ is the vector of observed data, and $e \in \mathbb{R}^{m}$ represents noise and other measurement errors. Solving (\ref{eq:ip}) is difficult since it belongs to a class of ill-posed inverse problems, in the sense that small changes in $b$ can produce large changes in the estimate of $x_{\rm true}$. This is due to the singular values of $A$ decaying and clustering at zero without any distinguishable gap between consecutive ones. For this reason, regularization must be implemented to diminish the instability invoked from the noise and the ill-posed nature of the problem in order to recover meaningful approximations of the solution \cite{Hansen2010}. 

A common approach for approximating $x_{\rm true}$ is iterative regularization. This consists of applying an iterative solver to the least squares problem 
\begin{equation} 
\label{eq:ir} 
\min_{x \in \mathbb{R}^n} \| b - Ax \|_2^{2}, 
\end{equation} 
where early termination produces a regularized solution \cite{chung2024computational}. The stopping iteration, which serves as a regularization parameter, is critical in constructing a solution that is not highly oscillatory or overly smooth. Alternatively, variational regularization can be used, where the aim is, for example,  to solve an optimization problem of the form: 
\begin{equation} 
\label{eq:vr}  
\min_{x \in \mathbb{R}^n} \| b - Ax \|_2^{2} + \lambda^2 \| x \|_2^{2},
\end{equation} 
where $\lambda$ is the regularization parameter and $ \| x \|_2^2$ is the regularization term \cite{chung2024computational}. Similar to the stopping iteration in iterative regularization, $\lambda$ must be chosen to prevent the approximate solution from becoming overly smooth if $\lambda$ is too large or highly oscillatory if $\lambda$ is too small.

Hybrid projection regularization combines iterative and variational techniques. This approach iteratively projects (\ref{eq:ir}) onto a small subspace with increasing dimension and applies variational regularization to the small projected problem. Since these are projection methods, they allow for explicit regularization while being `matrix-free', in the sense that they avoid explicit storage or construction of the system matrix as long as it is possible to efficiently compute matrix-vector products with $A$ (and possibly its transpose). This class of methods also provides a natural environment for estimating a good regularization parameter \cite{chung2024computational}.

Recent works on mixed-precision and highly parallel algorithms have led to an increased interest in the development of inner-product free methods. In the context of iterative methods for distributed memory machines, inner products require global communication and hence result in significant communication overhead \cite{crone1993communication,van2003iterative,de1995reducing,ghysels2013hiding,mcinnes2014hierarchical,lockhart2023performance}. The high communication costs for computing inner products may become a limiting factor in obtaining high speedups.  Moreover, different summation techniques may lead to varying accuracy of the computed inner product in floating point arithmetic \cite{higham1993accuracy,blanchard2020class}. Thus we are interested in the development of inner-product free iterative methods, specifically in the context of solving large-scale inverse problems.

\paragraph{Main Contributions} We introduce two new inner-product free iterative Krylov methods, one of which is a hybrid variant, for inverse problems with rectangular forward model matrices.  The approach is similar to recent work on the Hybrid Changing Minimal Residual Hessenberg Method \cite{brown2024h}, which was developed for problems where $A$ is square.  
We describe a new iterative projection approach called LSLU and a new hybrid projection algorithm called Hybrid LSLU that can be used for rectangular systems. Both are inner-product free Krylov methods, making them very appealing for mixed-precision arithmetic and parallel computing. We also show that the low-rank approximations computed during the iterative process can be used for efficient uncertainty quantification. Throughout the paper, we assume that $\| \cdot \|$ is the Euclidean norm unless otherwise indicated.

\section{Least Squares with LU Factorization}\label{sec:main}
In \Cref{sec:CHRM}, we review the Hessenberg Process, which serves as the backbone for the Changing Minimal Residual Hessenberg Method (CMRH) \cite{sadok1999new}, and describe how it can be used to approximately solve (\ref{eq:ir}) for square matrices $A$. Then in \Cref{sec:LSLU} we introduce the LSLU algorithm that extends CMRH to problems where $A$ is a rectangular matrix, which is a common scenario in the field of inverse problems.  We show that the Hessenberg process for rectangular matrices is directly  related to applying the Hessenberg process to the normal equations. In a similar fashion to CMRH, we impose a quasi-minimal residual optimality condition. This is comparable to the process found in LSQR \cite{paige1982lsqr} where the basis vectors are built using symmetric Lanczos on the normal equations, and the optimality conditions that are imposed minimize the residual norm.

\subsection{The Hessenberg Process and CMRH}\label{sec:CHRM}
CMRH iteratively approximates the solution to \eqref{eq:ir} in a Krylov subspace of increasing dimension,$$ {\cal{K}}_{k}(A,r_0) = span \{r_0, Ar_0,A^2r_0, \ldots,A^{k-1}r_0 \}, $$ where $A \in \mathbb{R}^{n \times n} $, $r_0 = b-Ax_0$, and $x_0$ is the initial guess of the solution. We define the Krylov matrix as
\begin{equation} \label{km}
V_{k} = [r_0, Ar_0,A^2r_0,\ldots,A^{k-1}r_0] \in \mathbb{R}^{n \times k},
\end{equation}
where $V_{k+1} = [r_0, AV_k]$.  Since $V_k$ is an ill-conditioned matrix for even small $k$, it is not explicitly constructed but only used to motivate the Hessenberg process.

For the Hessenberg process, consider the LU factorization,
$$ V_k = L_k U_k, $$ where $L_k \in \mathbb{R}^{n \times k}$ is a unit lower triangular matrix and $U_k \in \mathbb{R}^{k \times k}$ is an upper triangular matrix. The algorithm recursively computes the columns of $L_k$ but does not explicitly compute the LU factorization of $V_k$. From the construction of $V_{k+1}$, we can write the following relation,
\begin{equation} \label{eq:v_equ} V_{k+1} \begin{bmatrix} 0_{1 \times k} \\ I_k \end{bmatrix}  = L_{k+1} U_{k+1} \begin{bmatrix} 0_{1 \times k} \\ I_k \end{bmatrix} = AV_k = AL_k U_k, \end{equation} where $0_{1 \times k}$ is a row vector of zeros 
with dimensions $1 \times k$. Following \cite{sadok2012new}, we define an upper Hessenberg matrix as \begin{equation} \label{eq:H} 
H_{{k+1},k }= U_{k+1} \begin{bmatrix} 0_{1 \times k} \\ I_{k} \end{bmatrix}U_{k}^{-1} \in \mathbb{R}^{(k+1) \times k},
\end{equation} where $k<n$. Furthermore, combining (\ref{eq:v_equ}) and (\ref{eq:H}), we get the Hessenberg relation: \begin{equation}\label{eq:Hess1}
    AL_k = L_{k+1}H_{k+1,k},
\end{equation}
where the columns of $L_k$ form a linearly independent basis for ${\cal{K}}_k$.  \Cref{alg:hp1} contains a description of the Hessenberg Process for square $A$ \cite{sadok2008new}.

\begin{algorithm}[H]
\caption{Hessenberg Process for Square $A$} \label{alg:hp1}
\begin{algorithmic}[1]
\REQUIRE $A$, $b$, $x_0$, $\text{maxiter}$
\STATE $r_0=b-Ax_0$, $\beta = e_1^T r_0$; $l_1 = r_0 / \beta$
\FOR{$k = 1,\ldots,\text{maxiter}$} 
  \STATE $u = Al_k$
\FOR{$j = 1,\ldots,k$}
  \STATE $H(j,k) = u(j)$; $u = u-H(j,k)l_j$
  \ENDFOR
  \STATE $H(k+1,k) = u(k+1)$; $l_{k+1} = u/H(k+1,k)$
\ENDFOR 
\end{algorithmic}
\end{algorithm}

Notice that at each iteration of the Hessenberg process, we require one matrix-vector multiplication with $A$ and no computations of inner products.  From \Cref{alg:hp1}, we can observe that the Hessenberg process will break down if $\beta = e_1^T r_0= 0$ or $H(k+1,k) = 0$. To avoid this and to avoid severe ill-conditioning in the basis vectors, i.e, the columns of $L_k$, \cite{sadok2008new} introduces the Hessenberg process with pivoting, which is provided in \Cref{alg:hpp}. 

Building off the Hessenberg process, CMRH is an iterative projection algorithm for computing an approximate solution to (\ref{eq:ir}), where at each iteration $k$, the following least-squares problem is solved
\begin{equation}
\min_{x \in \range(L_k)} \| L_{k+1}^\dagger (b - A x) \|^2,
\end{equation}
where $\range(\cdot)$ is used to denote range of the given operator.
With initial guess $x_0$, $r_0 = b - A x_0$, and using \eqref{eq:Hess1}, it can be shown that the solution is given by $x_k = x_0 +L_k y_k$ where
\begin{equation}
\label{eq:CMRH_projected_problem}
y_k = \arg\min_{y \in \mathbb{R}^k} \| \beta e_1 - H_{k+1,k} y \|^2.
\end{equation}
Here, $\beta$ is either the first element of $r_0$ if using \Cref{alg:hp1} or the element  of $r_0$ with the highest absolute value if using \Cref{alg:hpp}.
A hybrid variant that incorporates regularization on the projected problem (\ref{eq:CMRH_projected_problem}) was described in \cite{brown2024h}.

\begin{algorithm}[H]
\caption{Hessenberg Process with Pivoting} \label{alg:hpp}
\begin{algorithmic}[1]
\REQUIRE $A$, $b$, $x_0$, $\text{maxiter}$
\STATE Define $p = [1,2,\ldots,n]^T$, and let $r_0 = b - A x_0$
\STATE Determine $i$ such that $|r_0(i)| = \|r_0\|_{\infty}$

\STATE $\beta = r_0 (i)$; $l_1 = r_0 / \beta$; $p(1)  \Leftrightarrow p(i)$
\FOR{$k = 1,\ldots, \text{maxiter}$} 
  \STATE $u = Al_k$
\FOR{$j = 1,\ldots,k$}
  \STATE $H(j,k) = u(p(j))$; $u = u-H(j,k)l_j$
  \ENDFOR
  \IF{$k<n$ and $u \neq 0$}
  \STATE Determine $i \in \{k+1,\ldots,n\}$ such that $|u(p(i))| = \|u(p(k+1:n))\|_\infty$
  \STATE $H(k+1,k) = u(p(i))$; $l_{k+1} = u/H(k+1,k)$ ; $p(k+1) \Leftrightarrow p(i)$
  \ELSE
  \STATE $H(k+1,k) = 0$; Stop
  \ENDIF
\ENDFOR 
\end{algorithmic}
\end{algorithm}

\subsection{Extension of the Hessenberg Process for rectangular systems and LSLU}
\label{sec:LSLU}
In this section, we describe an extension of the Hessenberg process for rectangular systems with $A \in \mathbb{R}^{m \times n}$, where the main difference is that we require two sets of basis vectors, one for each of the following Krylov subspaces,
\begin{eqnarray} 
{\cal{K}}_k(A^TA,v_0) = span\{v_0, A^TAv_0, (A^TA)^2v_0, \ldots, (A^TA)^{k-1}v_0 \},\label{eq:K1}  \\
{\cal{K}}_k(AA^T, r_0) = span \{ r_0, AA^Tr_0, (AA^T)^2r_0, \ldots,(AA^T)^{k-1}r_0 \}, \label{eq:K2} 
\end{eqnarray}
where $r_0 = b - Ax_0$ and $v_0 = A^Tr_0$, with $x_0$ an initial guess of the solution. 
Then, similar to CMRH, we introduce an iterative method called LSLU that minimizes an oblique projection of the residual and exploits components of the Hessenberg process for efficient computation.

Assume that no breakdowns occur in the initialization process. The Hessenberg method for rectangular systems, detailed in \Cref{alg:hpr}, generates at the $k$th iteration vectors $l_{k+1}$ and $d_{k+1}$ such that
\begin{eqnarray}
    AL_k &=& D_{k+1}H_{k+1,k}  \label{eq:nhr} \\A^TD_{k+1} &=& L_{k+1}W_{k+1}, \label{eq:nhr2} 
\end{eqnarray}
where $L_k \in \mathbb{R}^{n \times k}$ is unit lower triangular, $D_{k} \in \mathbb{R}^{m \times k}$ is unit lower triangular,  $H_{k+1,k} \in \mathbb{R}^{(k+1) \times k}$ is upper Hessenberg, and $W_{k+1} \in \mathbb{R}^{(k+1) \times (k+1)}$ is upper triangular.

From (\ref{eq:nhr}) and (\ref{eq:nhr2}), we obtain the following Hessenberg relationships:
\begin{eqnarray} \label{eq:h1} A^TAL_k = A^TD_{k+1}H_{k+1,k} = L_{k+1}W_{k+1}H_{k+1,k},\\
\label{eq:h2} AA^TD_{k+1} = AL_{k+1}W_{k+1} = D_{k+2}H_{k+2,k+1} W_{k+1} . \end{eqnarray} 
where the products $W_{k+1}H_{k+1,k}$ and $H_{k+2,k+1}W_{k+1}$ are upper Hessenberg matrices. Comparing \eqref{eq:h1} with \eqref{eq:Hess1}, we see that the Hessenberg process for rectangular systems is equivalent to the Hessenberg process for square systems applied to the normal equations, with system matrix $A^T A$ and resulting upper Hessenberg matrix $W_{k+1}H_{k+1,k}$.

\begin{algorithm}[H]
\caption{Hessenberg Process for Rectangular Systems} \label{alg:hpr}
\begin{algorithmic}[1]
\REQUIRE $A$, $b$, $x_0$, $\text{maxiter}$
\STATE Define $r_0 = b - Ax_0$, $\beta = e_1^T r_0$; $d_1 = r_0 / \beta$
\FOR{$k = 1,\ldots,\text{maxiter}$}
  \STATE $q = A^Td_{k}$
  \FOR{$j = 1,\ldots, k-1$}
  \STATE $W(j,k) = q(j)$; $q = q - W(j,k) l_j$
  \ENDFOR
  \STATE $W(k,k) = q(k)$; $l_{k} = q/W(k,k)$
  \STATE $u = Al_k$
\FOR{$j = 1,\ldots,k$}

  \STATE $H(j,k) = u(j)$; $u = u-H(j,k)d_j$
  \ENDFOR
  \STATE $H(k+1,k) = u(k+1)$; $d_{k+1} = u/H(k+1,k)$;  

\ENDFOR 
\end{algorithmic}
\end{algorithm}

Similar to the derivation in \Cref{sec:CHRM}, we can define Krylov matrices,
\begin{eqnarray}
P_k = [v_0, A^TAv_0, (A^TA)^2v_0,\ldots, (A^TA)^{k-1}v_0] \in \mathbb{R}^{n \times k}, \label{eq:Pk} \\
C_k = [r_0, AA^Tr_0, (AA^T)^2r_0,\ldots, (AA^T)^{k-1}r_0] \in \mathbb{R}^{m \times k}, \label{eq:Ck}
\end{eqnarray}
 whose columns span (\ref{eq:K1}) and (\ref{eq:K2}) respectively. It follows that $P_{k+1} = [v_0, A^TAP_k]$  and $C_{k+1} = [r_0 , AA^TC_k]$. 

Note that, by construction, the columns of $P_k$ and $L_k$ span the same space for all $k$. In particular, the vector $p_j$ can be written as a linear combination of the columns of $L_j$, which correspond to the first $j$ columns of the matrix $L_k$, for all $j\leq k$. This means,  there exists an upper triangular matrix $U_k$ such that $P_k = L_k U_k$, and since $L_k$ is unit lower triangular, this corresponds to an LU factorization of $P_k$. Note that \Cref{alg:hpr} does not explicitly compute this LU factorization, but recursively generates the columns of $L_k$. Applying this factorization provides the following relation:
\begin{equation} \label{eq:Pkrelation} P_{k+1} \begin{bmatrix} 0_{1 \times k} \\ I_k \end{bmatrix} = L_{k+1}U_{k+1} \begin{bmatrix} 0_{1 \times k} \\ I_k \end{bmatrix} = A^TAP_k = A^TAL_kU_k,  \end{equation}
where $0_{1 \times k}$ is a row vector of zeros with dimensions $1 \times k$. Thus, given \eqref{eq:h1}, we recover the upper Hessenberg matrix: 
\begin{equation} \label{eq:hequation} W_{k+1}H_{k+1,k} = U_{k+1} \begin{bmatrix} 0_{1 \times k} \\ I_k \end{bmatrix} U^{-1}_{k}, \end{equation} where $k < n$.

Following an analogous argument to the one used for $L_k$, the columns of $C_k$ and $D_k$ span the same space for all $k$ by construction, and there exists an upper triangular matrix $G_{k+1}$ such that $C_{k+1} = D_{k+1} G_{k+1}$ corresponds to an LU factorization of $C_{k+1}$. Applying this factorization to $C_{k+1} = [r_0 , AA^TC_k]$ provides the following relation:
\begin{equation} \label{eq:Ckrelation} C_{k+2} \begin{bmatrix} 0_{1 \times (k+1)} \\ I_{k+1} \end{bmatrix} = D_{k+2}G_{k+2} \begin{bmatrix} 0_{1 \times (k+1)} \\ I_{k+1} \end{bmatrix} = AA^TC_{k+1} = AA^TD_{k+1}G_{k+1},  \end{equation}
where $0_{1 \times (k+1)}$ is a row vector of zeros with dimensions $1 \times (k+1)$. Thus, given \eqref{eq:h2}, we recover the upper Hessenberg matrix: 
\begin{equation} \label{eq:hwequation} H_{k+2, k+1}W_{k+1} = G_{k+2} \begin{bmatrix} 0_{1 \times k+1} \\ I_{k+1} \end{bmatrix} G^{-1}_{k+1}, \end{equation} where $k < m$.

From \Cref{alg:hpr}, we find that the process will breakdown if either $\beta = 0$, $H(k+1,k)= 0$, or $W(k,k) = 0$. To avoid this, in practice we implement the Hessenberg process with pivoting instead, which is given in \Cref{alg:hpr2}.
\begin{algorithm}[H]
\caption{Hessenberg Process with Pivoting for Rectangular Systems} \label{alg:hpr2}
\begin{algorithmic}[1]
\REQUIRE $A$, $b$, $x_0$, $\text{maxiter}$
\STATE Define $t = [1,2,\ldots,m]^T$, $g = [1,\ldots,n]^T $.
\STATE $r_0 = b - A x_0 $ 
\STATE Determine $i$ such that $|r_0(i)| = \|r_0\|_{\infty}$
\STATE $\beta = r_0(i)$; $d_1 = r_0/ \beta$; $t(1) \Leftrightarrow t(i)$

\FOR{$k = 1,\ldots,\text{maxiter}$} 
  \STATE $q = A^Td_{k}$
  \FOR{$j = 1,\ldots, k-1$}
  \STATE $W(j,k) = q(g(j))$; $q = q - W(j,k)l_j$
  \ENDFOR
  \IF{$k<n$ and $q \neq 0$}
  \STATE Determine $i \in \{ k,\ldots,n\}$ such that $|q(g(i))| = \|q(g(k:n))\|_{\infty}$
  \STATE $W(k,k) = q(g(i))$; $l_{k} = q/W(k,k)$; $g(k) \Leftrightarrow g(i)$
  \ELSE
   \STATE break
  \ENDIF 
  \STATE $u = Al_k$
\FOR{$j = 1, \ldots,k$}
  \STATE $H(j,k) = u(t(j))$; $u = u-H(j,k)d_j$
  \ENDFOR
  \IF{$k<m$ and $u \neq 0$}
  \STATE Determine $i \in \{ k+1,\ldots,m\}$ such that $|u(t(i))| = \|u(t(k+1:m))\|_{\infty}$
  \STATE $H(k+1,k) = u(t(i))$; $d_{k+1} = u/H(k+1,k)$; $t(k+1) \Leftrightarrow t(i)$
  \ELSE
 \STATE break
  \ENDIF   
\ENDFOR 
\end{algorithmic} 
\end{algorithm}

LSLU is a new iterative projection method that, at each iteration $k$, finds an approximate solution for (\ref{eq:ir}) by minimizing the following least squares problem:
\begin{equation} \label{eq:quasi} \min_{x \in x_0+\range(L_k)}  \| D^{\dagger}_{k+1} (b-Ax) \|, \end{equation}
where $D^{\dagger}_{k+1}$ is the pseudoinverse of $D_{k+1}$. Note that the functional in (\ref{eq:quasi}) can be considered as an approximation to the residual norm of the original problem, similarly to the QMR method. More specifically, considering $x = x_0 + L_ky$ and $r_0 = b - Ax_0$, the objective function in (\ref{eq:quasi}) can be written as
\begin{eqnarray*}
  \| D^{\dagger}_{k+1} (b - A(x_0 + L_k y)) \| 
 &=& \|D^{\dagger}_{k+1} (r_0 - A L_k y) \| \\
  &=&  \| D^{\dagger}_{k+1} (r_0 - D_{k+1} H_{k+1,k}y) \| \\
   &=&  \| \beta e_{1} - H_{k+1,k}y \|, 
 \end{eqnarray*}
where $\beta$ is either the first entry of $r_0$ (\Cref{alg:hpr}) or the entry of $r_0$  with the highest absolute value (\Cref{alg:hpr2}). Thus, at iteration $k$ we solve the following subproblem, $$ y_k = \arg \min_{y \in \mathbb{R}^k} \| \beta e_1 - H_{k+1,k}y \|, $$
which is of much smaller dimension compared to the original problem. Once $y_k$ is computed, then $x_k = x_0 + L_k y_k$ provides an approximate solution of the original least squares problem (\ref{eq:ir}). The algorithm corresponding to this method can be found as a special case of the hybrid method described in \Cref{sec:hybrid}.

In \Cref{alg:hpr2}, we must find the entry with the highest absolute value of $r_0$ and two other vectors at each iteration. For any given vector $x$, this correspond to finding  $i$ such that $|x(i)| = \|x\|_{\infty}$, which can be costly as computing $\|x\|_{\infty}$ requires global communication. In order to avoid this in LSLU, we also propose the following pivoting alternative: select a small random sample of entries from $r_0, u, q$, and choose the largest value (in magnitude) in that sample. Provided that the selection is ``large enough'', we achieve a reasonable approximate solution.  

As an illustration, we use the the PRtomo example from IR Tools \cite{gazzola2018ir} (see \Cref{sec:num} for details), and we use $25, 50,$ and $100$ samples to approximate the infinity norm.  Note that the samples are only being used for determining the pivot, and that the number of samples is tiny compared to the more than $65000$ elements in each of the vectors. We provide relative reconstruction error norms per iteration in \Cref{fig:PRtom}, where the sampled LSLU approach denoted `LSLU inf est' performs similar to the LSLU approach where the pivots are determined using the actual infinity norm, denoted `LSLU'. Note that sometimes `LSLU inf est' seems to perform better than `LSLU' in that one can observe a delay in the semi-convergence phenomenon.  However, the minimal attained error norm for `LSLU inf est' is comparable or marginally larger than the one corresponding to the version with standard partial pivoting.  The relative reconstruction error norms per iteration of LSQR are provided to illustrate that the new LSLU method is competitive with existing methods.  Additional numerical results will be provided in \Cref{sec:num}.

\begin{figure}[h]
\centering
\begin{tabular}{ccc}
    {\scriptsize PRtomo Sample Size $25$} &  {\scriptsize PRtomo Sample Size $50$} & {\scriptsize PRtomo Sample Size $100$} \\ 
   \includegraphics[width=3.92cm]{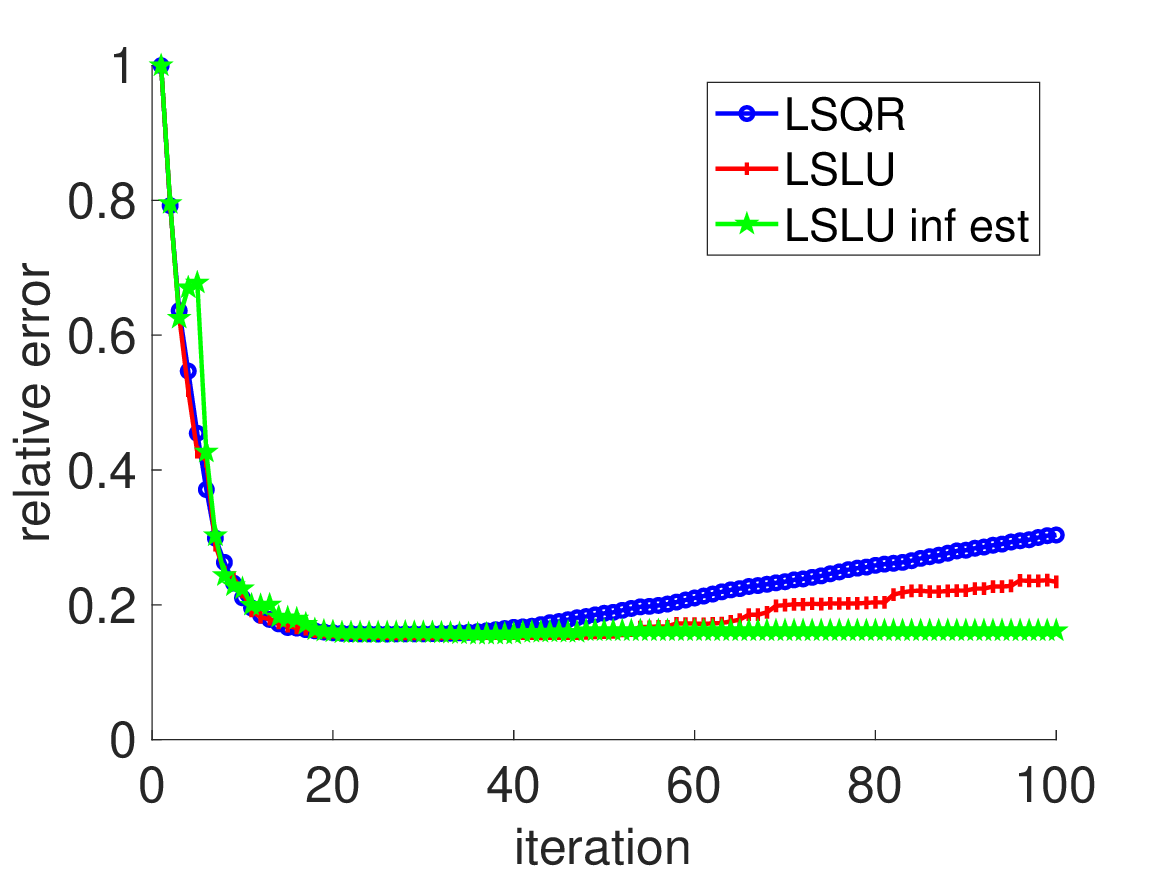} &  \includegraphics[width=3.92cm] {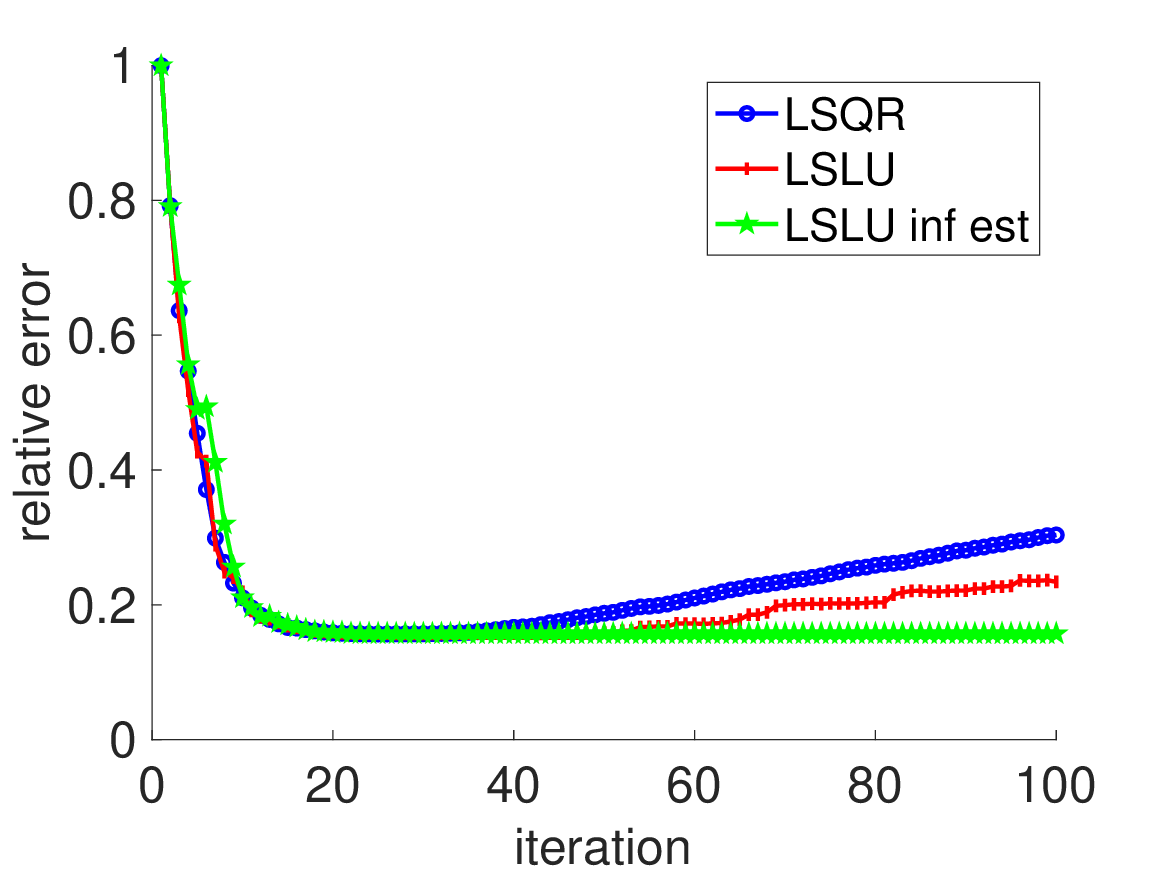} &
   \includegraphics[width=3.92cm]{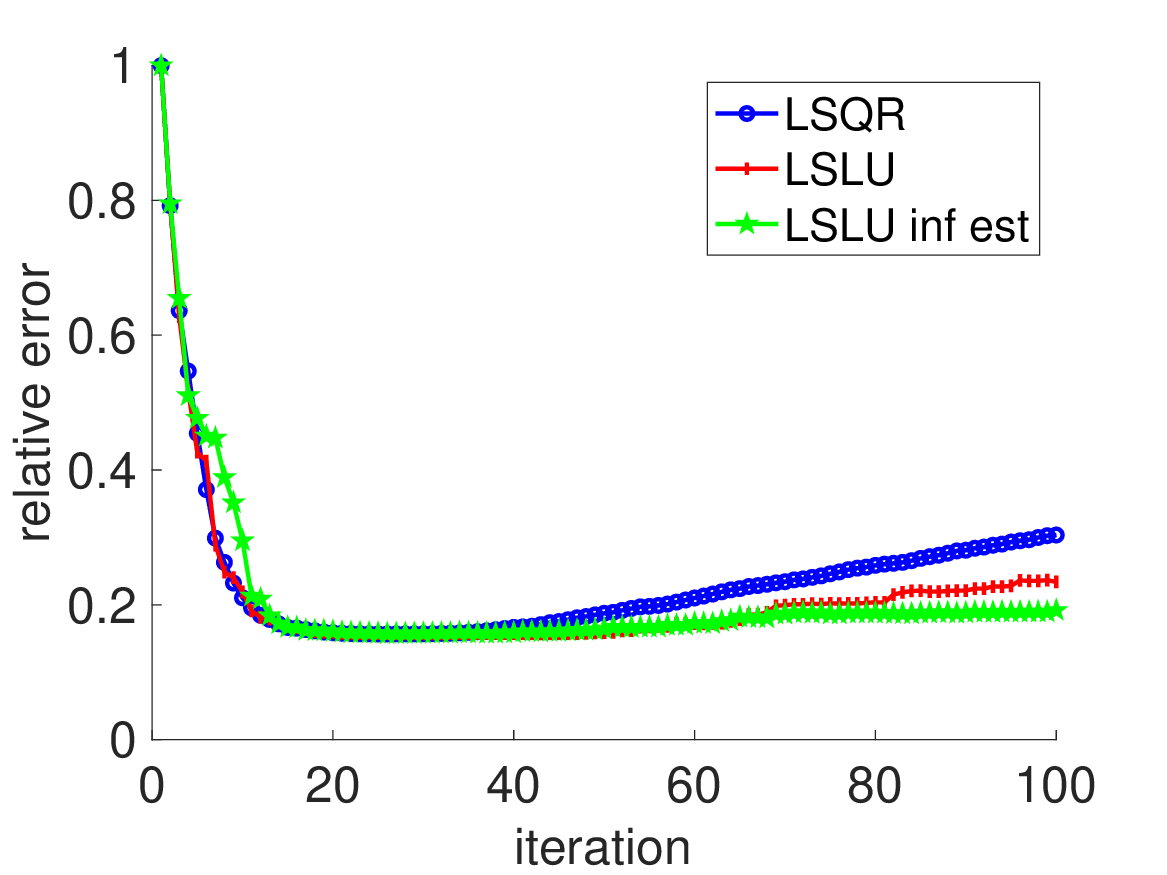} \\
    \end{tabular}
\caption{Relative reconstruction error norms per iteration for LSLU with pivoting using the infinity norm, compared to using the estimated infinity norm as the maximum from a set of randomly sampled coefficients (denoted `LSLU inf est').  Results for LSQR are provided for reference.}
\label{fig:PRtom}
\end{figure}

\subsection{Theoretical bounds for the residual norm of LSLU}
In this section, we derive a bound on the difference between the residual norms of solutions computed using LSLU and LSQR. Let $\hat{R}_{k+1}$ be an upper triangular matrix from the QR decomposition of $D_{k+1}$. We show that if the condition number of $\hat{R}_{k+1}$ does not grow too quickly, the residual norms associated with the approximate solutions of LSLU and LSQR at each iteration are close to each other. This is critical in understanding the regularizing properties of LSLU.
\begin{theorem}
Let $r_k^{QR}$ and $r_k^{LU}$ be the LSQR and LSLU residuals at the kth iteration beginning with the same initial guess $x_0 = 0$, respectively. Then 
\begin{equation}\label{eq:eqthm} \|r_k^{QR}\| \leq \|r_k^{LU}\| \leq \kappa (\hat{R}_{k+1}) \|r_k^{QR} \| \end{equation} where $\kappa (\hat{R}_{k+1}) = \|\hat{R}_{k+1}\| \|\hat{R}_{k+1}^{-1}\|$ is the condition number of $\hat{R}_{k+1}$.
\end{theorem}

\begin{proof}
    First, we prove the left inequality in (\ref{eq:eqthm}). Consider the residual as a function of the solution, $$ r(x) = b - Ax.$$ Then, the residual norm associated with the approximate solution at the $k$th iteration of LSQR is $$ \|r_k^{QR}\| = \| b-Ax_{k}^{QR}\| = \min_{x \in {\cal{K}}_k(A^TA,A^Tb)} \|r(x)\|. $$ Since $x_k^{QR}$ and $x_k^{LU}$ are in the Krylov subspace ${\cal{K}}_k(A^TA,A^Tb)$, then by definition,
    $$ \min_{x \in {\cal{K}}_k(A^TA,A^Tb)} \|r(x)\| \leq \|r(x_k^{LU})\| = \|r_k^{LU}\|. $$ Hence, $\|r_k^{QR}\| \leq \|r_k^{LU}\|$.

    Now we prove the right inequality in (\ref{eq:eqthm}). Since $r_k^{QR}$ and $r_k^{LU}$ are in the subspace ${\cal{K}}_k(AA^T,b)$, we can write $r_k^{QR}$ and $r_k^{LU}$ as a linear combination of any basis of ${\cal{K}}_k(AA^T,b)$. Using the Hessenberg relation, the LU decompositions of $P_{k+1}$ and $C_{k+1}$ are
    $$C_{k+1} = D_{k+1} G_{k+1}$$
    $$P_{k+1} = L_{k+1} U_{k+1}. $$
    This implies that $\range(D_{k+1}) = 
    \range(C_{k+1}) = {\cal{K}}_{k+1}(AA^T,b)$ and $\range(L_{k+1}) = \range(P_{k+1}) = {\cal{K}}_{k+1}(A^TA,A^Tb)$. Therefore, using the QR decomposition of $D_{k+1} = \hat{U}_{k+1} \hat{R}_{k+1}$ and $L_{k+1} = V_{k+1}\Tilde{R}_{k+1}$, there exist $z_k^{LU}$ and $s_k^{LU}$ in $\mathbb{R}^{k+1}$ such that  \begin{equation} \label{eq:lu_r} r_k^{LU} = D_{k+1}z_k^{LU} = \hat{U}_{k+1} \hat{R}_{k+1}z_k^{LU} = \hat{U}_{k+1}s_k^{LU} \end{equation} with $ s_k^{LU} = \hat{R}_{k+1}z_k^{LU}$. Analogously, there exist $z_k^{QR}$ and $s_k^{QR}$ in $\mathbb{R}^{k+1}$ such that \begin{equation} \label{eq:qr_r} r_k^{QR} = D_{k+1}z_k^{QR} = \hat{U}_{k+1} \hat{R}_{k+1}z_k^{QR} = \hat{U}_{k+1}s_k^{QR} \end{equation} with $s_k^{QR} = \hat{R}_{k+1}z_k^{QR}$.
    
Consider the optimality conditions of LSLU. As stated above,
$r_k^{LU} = D_{k+1}z_k^{LU}$. This implies that $D^{\dagger}_{k+1} r_k^{LU} = z_k^{LU}$. Hence, $\| D^{\dagger}_{k+1} r_k^{LU}\| = \|z_k^{LU}\|$ so
\begin{equation} \label{eq:lsluthm}
    \|z_k^{LU}\| = \min_{x \in \range(L_k)} \|D^{\dagger}_{k+1}(b-Ax)\| = \min_{x \in \range(L_k)} \|D^{\dagger}_{k+1} r(x)\|
\end{equation}
Using (\ref{eq:lsluthm}) and the fact that $x_k^{QR}$ is in ${\cal{K}}_k(A^TA,A^Tb)$ then $$ \|z_k^{LU}\| = \min_{x \in {\cal{K}}_k(A^TA,A^Tb)} \|D^{\dagger}_{k+1} r(x)\| \leq \|D^{\dagger}_{k+1} r(x_k^{QR})\| = \|z_k^{QR}\|. $$Thus $$ \|z_k^{LU}\| \leq \|z_k^{QR}\| = \| \hat{R}^{-1}_{k+1}s_k^{QR} \| \leq \| \hat{R}^{-1}_{k+1}\| \|s_k^{QR} \| = \| \hat{R}^{-1}_{k+1}\| \|r_k^{QR} \|, $$ where the equalities in the above relation come from (\ref{eq:qr_r}). On the other hand, applying (\ref{eq:lu_r}) gives: $$\|r_k^{LU}\| = \|D_{k+1}z_k^{LU}\| \leq \|D_{k+1}\| \|z_k^{LU}\|.$$

Putting the above inequalities together gives the following relation:
\begin{eqnarray*}
  \| r_k^{LU} \| 
 &=& \|D_{k+1} z_k^{LU} \| \\
  &\leq&  \| D_{k+1} \|  \|z_k^{LU}\| \\
   &\leq&  \| D_{k+1} \| \| \hat{R}^{-1}_{k+1} \| \|r_k^{QR} \|. 
 \end{eqnarray*}
Recall that $D_{k+1}$ has a QR decomposition of the form  $D_{k+1} = \hat{U}_{k+1} \hat{R}_{k+1}$, where $ \hat{U}_{k+1}$ is an orthogonal matrix. Therefore,  $\|D_{k+1}\| = \| \hat{U}_{k+1} \hat{R}_{k+1}\| = \| \hat{R}_{k+1}\|$. This results in the following: $$ \|r_k^{LU}\| \leq \|D_{k+1}\| \|\hat{R}^{-1}_{k+1}\| \|r_k^{QR} \| = \|\hat{R}_{k+1}\| \|\hat{R}^{-1}_{k+1}\| \|r_k^{QR} \| = \kappa(\hat{R}_{k+1})  \|r_k^{QR} \|.$$ 
Thus, we conclude that $ \|r_k^{QR}\|  \leq \|r_k^{LU}\| \leq \kappa(\hat{R}_{k+1})  \|r_k^{QR} \|$.
\end{proof}

\section{Hybrid LSLU} \label{sec:hybrid}
In this section, we consider a hybrid variant of LSLU for solving large-scale linear inverse problems. In addition to being inner product free, this method can compute regularized solutions efficiently and with automatically selected regularization parameters. In \Cref{sec:theory_h} we provide theoretical bounds for the residual norms of Hybrid LSLU, and in \Cref{sec:computing} we address some computational considerations.

Consider the standard Tikhonov regularization problem \eqref{eq:vr}.
The Hybrid LSLU 
method is an iterative method, where the solution at the $k$th iteration is computed as the solution to the optimization problem,
 
\begin{equation} \label{eq:TP} 
\min_{x \in x_0+\range(L_k)} \| D^{\dagger}_{k+1}(b-Ax)\|^2 + \lambda^2_k \| L^{\dagger}_{k} x \|^2, 
\end{equation}
where similar to LSLU, the residual norm is replaced by a semi-norm, and the regularization term also includes a semi-norm. It can be shown that solving (\ref{eq:TP}) is equivalent to solving 
\begin{equation} \label{eq:TSP} 
y_{\lambda ,k} = \arg \min_{y \in \mathbb{R}^k} \| \beta e_1 - H_{k+1,k}y \|^2 +  \lambda^2_k \|y\|^2, 
\end{equation} 
where $\beta$ is the largest entry in $r_0$ (when considering the Hessenberg method implementation with pivoting), and we can project back onto the original subspace using $x_k = x_0 + L_ky_{\lambda ,k}$. An implementation of Hybrid LSLU with pivoting is provided in \Cref{alg:HLSLU}, which also corresponds to LSLU if one sets $\lambda$ = 0. 
As a hybrid approach, the regularization parameter, denoted as $\lambda_k$ in (\ref{eq:TP}) and (\ref{eq:TSP}), can be selected at each iteration. This will be discussed in \Cref{sec:paramselect}.

\begin{algorithm}[H]
\caption{Hybrid LSLU} \label{alg:HLSLU}
\begin{algorithmic}[1]
\REQUIRE $A$, $b$, $x_0$, $\text{maxiter}$, $\text{RegParam}$
\STATE Define $t = [1,2,\ldots,m]^T$, $g = [1,\ldots,n]^T $.
\STATE $r_0 = b - A x_0 $ 
\STATE Determine $i$ such that $|r_0(i)| = \|r_0\|_{\infty}$
\STATE $\beta = r_0(i)$; $d_1 = r_0/ \beta$; $t(1) \Leftrightarrow t(i)$

\FOR{$k = 1,\ldots,\text{maxiter}$} 
  \STATE $q = A^Td_{k}$
  \FOR{$j = 1,\ldots, k-1$}
  \STATE $W(j,k) = q(g(j))$; $q = q - W(j,k)l_j$
  \ENDFOR
  \IF{$k<n$ and $q \neq 0$}
  \STATE Determine $i \in \{ k,\ldots,n\}$ such that $|q(g(i))| = \|q(g(k:n))\|_{\infty}$
  \STATE $W(k,k) = q(g(i))$; $l_{k} = q/W(k,k)$; $g(k) \Leftrightarrow g(i)$
  \ELSE
   \STATE break
  \ENDIF 
  \STATE $u = Al_k$
\FOR{$j = 1, \ldots,k$}
  \STATE $H(j,k) = u(t(j))$; $u = u-H(j,k)d_j$
  \ENDFOR
  \IF{$k<m$ and $u \neq 0$}
  \STATE Determine $i \in \{ k+1,\ldots,m\}$ such that $|u(t(i))| = \|u(t(k+1:m))\|_{\infty}$
  \STATE $H(k+1,k) = u(t(i))$; $d_{k+1} = u/H(k+1,k)$; $t(k+1) \Leftrightarrow t(i)$
  \ELSE
 \STATE break
  \ENDIF   
  \STATE Find regularization parameter $\lambda_k$ according to the $\text{RegParam}$ scheme. 
  \STATE Compute $y_{\lambda_k,k}$ being to the minimizer of $\| \beta e_1 - H_{k+1,k}y \|_2^2 + \lambda_k^2 \|y\|_2^2$ 
  \STATE $ x_k = x_0 + L_k y_{\lambda_k,k}$
\ENDFOR 
\end{algorithmic} 
\end{algorithm}

\subsection{Theoretical bounds for the residual norms for Hybrid LSLU}
\label{sec:theory_h}
Similar to LSLU and LSQR, the residual norms of Hybrid LSLU
and Hybrid LSQR can be bounded in an analogous fashion. These bounds provide insight on the regularizing properties of Hybrid LSLU. Let $\lambda$ be fixed and let
\begin{equation}
\label{eq:Dbar}
\overline{D}_{k+1} = \begin{bmatrix} D_{k+1} & 0 \\ 0 & L_{k} \end{bmatrix}, 
\end{equation}
with $D_{k+1}$ and $L_{k}$ defined by the Hessenberg relations (\ref{eq:nhr}) and (\ref{eq:nhr2}). We find that if the condition number of $\overline{D}_{k+1}$ does not grow too quickly, then the residual norm associated to the solution obtained with Hybrid LSLU is close to the residual norm of the solution obtained with Hybrid LSQR.
\begin{theorem}\label{thm:bounds}
Let $hr_k^{QR}$ and $hr_k^{LU}$ be the Hybrid LSQR and Hybrid LSLU residuals at the kth iteration beginning with the same initial residual $r_0$, respectively. Then 
\begin{equation}\label{eq:eqthm2} \|hr_k^{QR}\| \leq \|hr_k^{LU}\| \leq \kappa (\overline{D}_{k+1}) \|hr_k^{QR} \| \end{equation} where $\kappa (\overline{D}_{k+1}) = \|\overline{D}_{k+1}\| \|\overline{D}_{k+1}^{\dagger}\|$ is the condition number of $\overline{D}_{k+1}$.
\end{theorem}

\begin{proof}
First, we prove the left inequality in (\ref{eq:eqthm2}). We can define the hybrid residual as a function of the solution, $$ hr(x) = \begin{bmatrix}
    b \\ 0
\end{bmatrix}  - \begin{bmatrix}
   A \\ \lambda I 
\end{bmatrix}x.$$
Since $x_k^{QR}$ and $x_k^{LU}$ are in the Krylov subspace ${\cal{K}}_k(A^TA,A^Tb)$, by the optimality conditions of Hybrid LSQR, $$ \|hr_k^{QR}\| = \min_{x \in \range(L_k)}\|hr(x)\| \leq \| hr(x_k^{LU})\| = \|hr_k^{LU}\|.$$ 
Hence $\|hr_k^{QR}\| \leq \|hr_k^{LU}\|$.

Now we prove the right inequality in (\ref{eq:eqthm2}). Since $b-Ax \in {\cal{K}}_{k+1}(AA^T,b) = \range(D_{k+1}$), then for any $x \in {\cal{K}}_k(A^TA,A^Tb)$ and $x_k^{QR}$, $x_k^{LU} \in  {\cal{K}}_k(A^TA,A^Tb) = \range(L_{k}$) 
we can write $hr_k^{LU}$ and $hr_k^{QR}$ as a linear combination of the columns of $\overline{D}_{k+1}$ defined in (\ref{eq:Dbar}).

Let $hr_k^{LU} = \overline{D}_{k+1} z_k^{LU}$ and $hr_k^{QR} =\overline{D}_{k+1} z_k^{QR}$. This implies that $z_k^{LU} = \overline{D}^{\dagger}_{k+1} hr_k^{LU}$ and $z_k^{QR} = \overline{D}^{\dagger}_{k+1} hr_k^{QR}$. Hence, $\|\overline{D}^{\dagger}_{k+1} hr_k^{LU}\| = \|z_k^{LU}\|$ and $\|\overline{D}^{\dagger}_{k+1} hr_k^{QR}\| = \|z_k^{QR}\|$. By the optimality conditions of Hybrid LSLU,
\begin{equation} \label{eq:opt}
    \|z_k^{LU}\| = \min_{x \in \range(L_k)} \left\| \begin{bmatrix}
        D^{\dagger}_{k+1} & 0 \\ 0 & L^{\dagger}_k
    \end{bmatrix} \left(\begin{bmatrix}
        b \\ 0
    \end{bmatrix} - \begin{bmatrix}
        A \\ \lambda I
    \end{bmatrix}\right)x \right \| = \min_{x \in \range(L_k)} \| \overline{D}^{\dagger}_{k+1} hr(x)\|.
\end{equation}
Using (\ref{eq:opt}) and the fact that $x_k^{QR}$ is in ${\cal{K}}_k(A^TA,A^Tb)$, $$ \|z_k^{LU}\| = \min_{x \in \range(L_k)} \| \overline{D}^{\dagger}_{k+1} hr(x)\| \leq \| \overline{D}^{\dagger}_{k+1} hr(x_k^{QR})\| = \|z_k^{QR}\|. $$Thus $$\|z_k^{LU}\| \leq \|z_k^{QR}\| = \| \overline{D}^{\dagger}_{k+1} hr_k^{QR} \| \leq  \| \overline{D}^{\dagger}_{k+1} \| \|hr_k^{QR} \|.$$

Putting the above inequalities together gives the following relation,
\begin{eqnarray*}
  \| hr_k^{LU} \| 
 &=& \|\overline{D}_{k+1} z_k^{LU} \| \\
  &\leq&  \| \overline{D}_{k+1} \|  \|z_k^{LU}\| \\
   &\leq&  \| \overline{D}_{k+1} \| \| \overline{D}^{\dagger}_{k+1} \| \|hr_k^{QR} \|, 
 \end{eqnarray*}
 so we conclude that $\|hr_k^{QR}\| \leq \|hr_k^{LU}\| \leq \kappa (\overline{D}_{k+1}) \|hr_k^{QR}\|$.
\end{proof}

To illustrate the behavior of the residual norms for Hyrbid LSLU and Hybrid LSQR as well as to investigate the bound in \Cref{thm:bounds}, we plot in \Cref{fig:PRresidual} the residual norms per iteration for three different test problems: PRtomo, PRspherical, and PRseismic from the IR tools package \cite{gazzola2018ir}. We fix  $\lambda= 0.01$ and plot residual norms for Hybrid LSLU along with the lower and upper bounds from \Cref{thm:bounds}. We observe that the residual norms for Hybrid LSLU and Hybrid LSQR remain close for PRtomo and PRspherical.  As expected, the residual norms for solutions computed using Hybrid LSQR provide a lower bound for residual norms for solutions computed using Hybrid LSLU. The upper bound from \Cref{thm:bounds} given by $\kappa (\overline{D}_{k+1}) \|hr_k^{QR}\|$ becomes looser with more iterations. For details regarding the test problems, see \Cref{sec:num}.

\begin{figure}[h]
\centering
\begin{tabular}{ccc}
    {\scriptsize PRtomo} &  {\scriptsize PRspherical} & {\scriptsize PRseismic} \\ 
   \includegraphics[width=3.92cm]{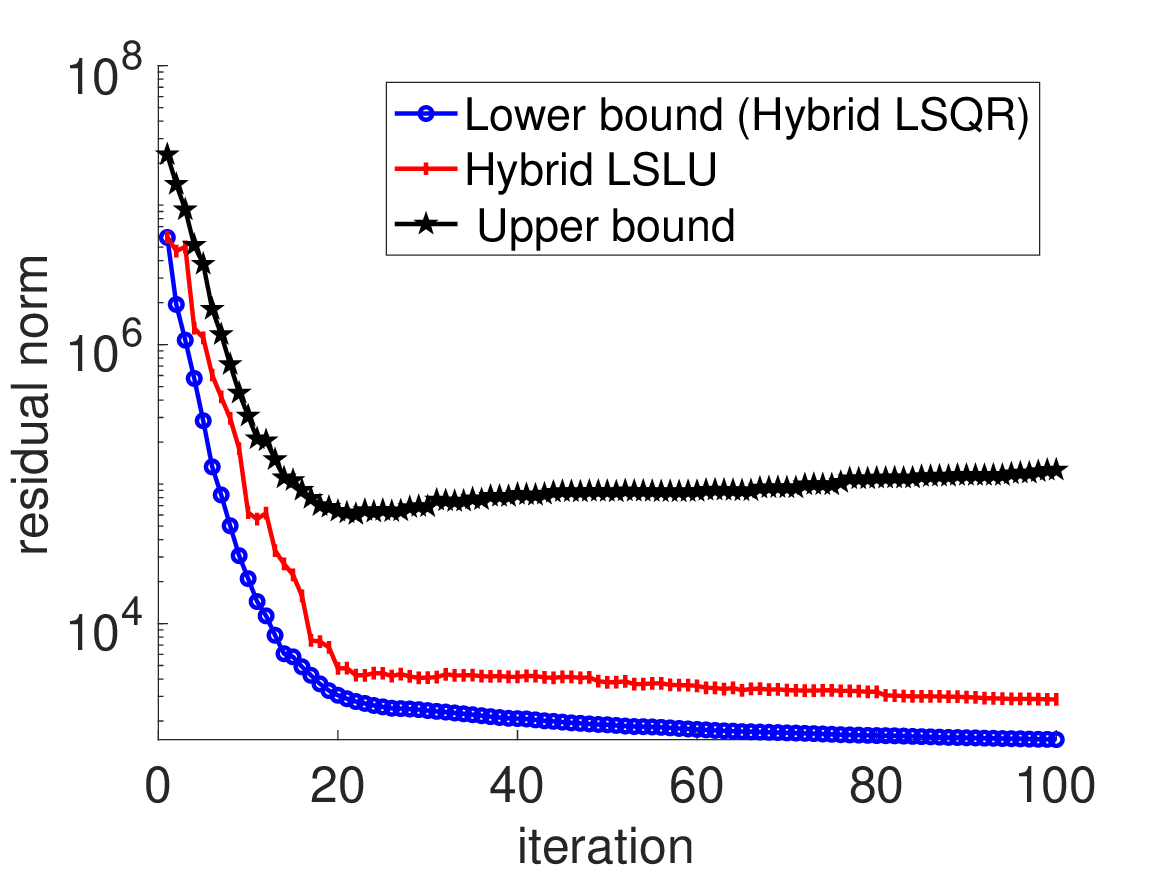} &  \includegraphics[width=3.92cm] {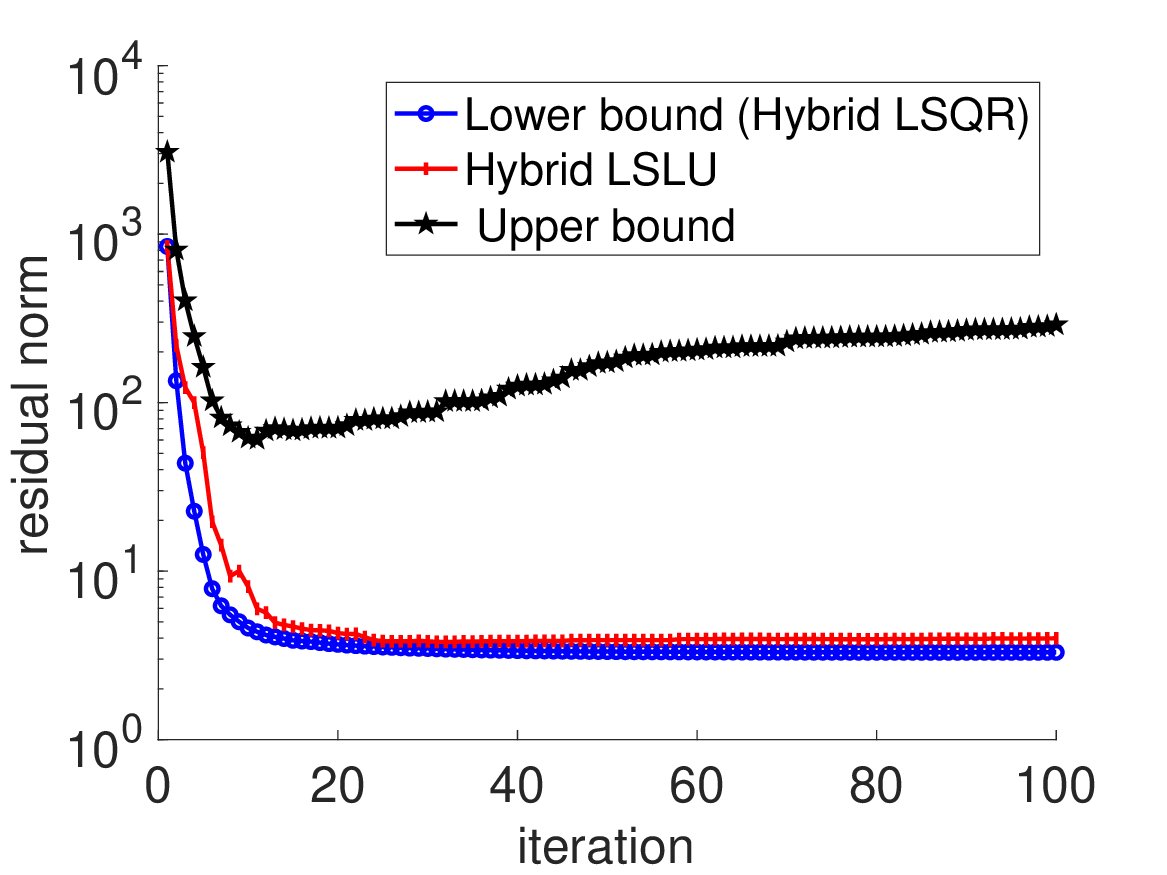} &
   \includegraphics[width=3.92cm]{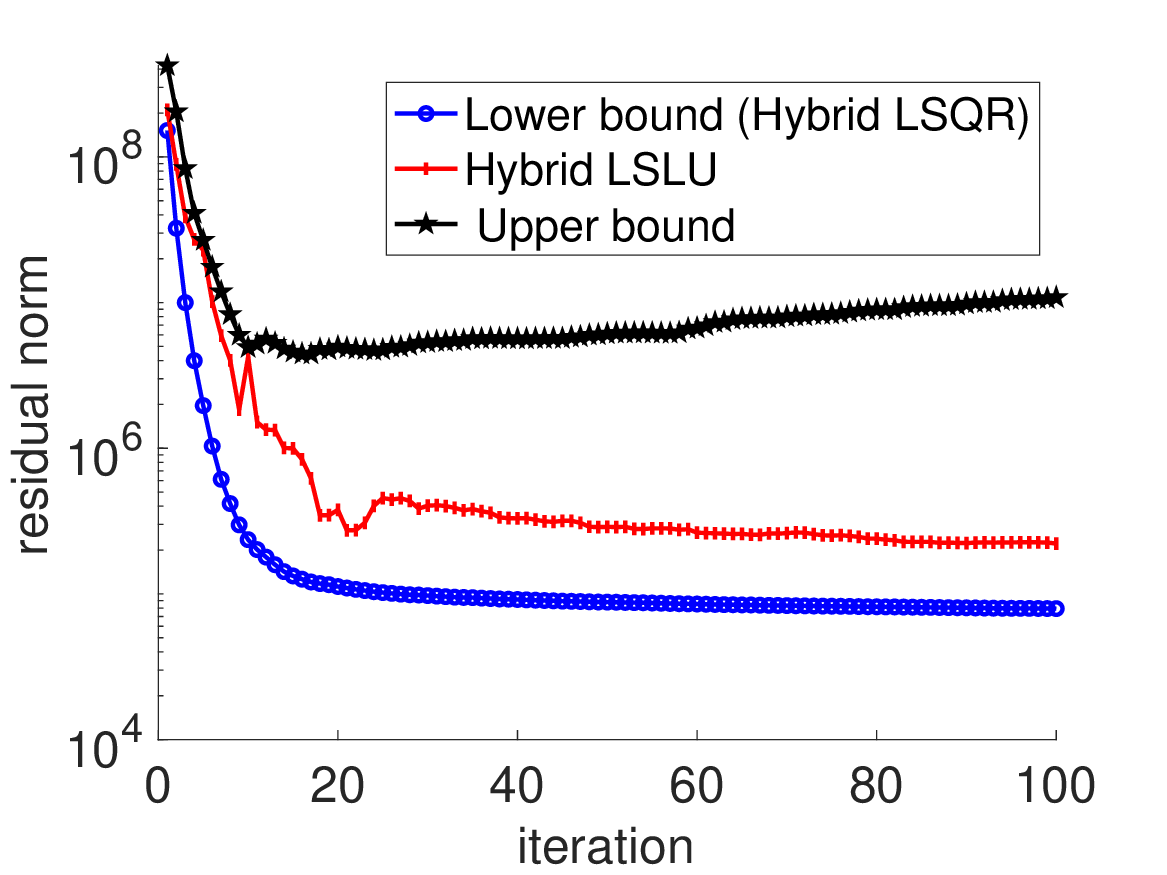} \\
    \end{tabular}
    
\caption{Residual norms per iteration for Hybrid LSLU, as well as corresponding bounds from \Cref{thm:bounds}. Note that the lower bound corresponds to Hybrid LSQR residual norms.}
\label{fig:PRresidual}
\end{figure}

\subsection{Computational Considerations}
\label{sec:computing}
In this section we describe some of the computational aspects of Hybrid LSLU. In particular, we describe methods for selecting regularization parameter $\lambda_k$ at each iteration in \Cref{sec:paramselect} and stopping criterion in \Cref{sec:stopping}.

\subsubsection{Selecting Regularization Parameters}
\label{sec:paramselect}
Our objective is to find an appropriate regularization parameter $\lambda_k$ at each iteration of Hybrid LSLU that will not cause the regularized solution $x_k$ to be overly oscillatory or too smooth. There are various approaches for selecting regularization parameters within hybrid projection methods \cite{chung2024computational}.  We consider Tikhonov regularization for the projected problem (\ref{eq:TSP}). SVD based approaches can be used to find a good estimate for the regularization parameter $\lambda_k$, since the projected problem (\ref{eq:TSP}) is significantly smaller than (\ref{eq:TP}).  

Although not available in practice, we compute the optimal regularization parameter for simulated data to demonstrate the potential benefits of Hybrid LSLU. The optimal regularization parameter requires knowledge of the true solution and is obtained by minimizing the following expression:
\begin{equation} \label{eq:lambda}  
 \lambda_k =  \arg \min_{\lambda} \| x_{\lambda,k} - x_{\rm true}\|_2^2, \end{equation}
where $x_{\lambda,k}$ is the approximate solution at the $k$th iteration with the regularization parameter $\lambda$. Assume that $x_{\lambda,k} = x_0 + L_ky_{\lambda,k}$. Then (\ref{eq:lambda}) can be rewritten as:
\begin{equation}\label{eq:1}
 \min_{\lambda} \|x_{\lambda,k} - x_{\rm true}\|_2^{2} = \min_{\lambda} \| x_0 + L_ky_{\lambda,k} - x_{\rm true}\|_{2}^{2}.
\end{equation}
From  (\ref{eq:1}), we replace $y_{\lambda,k}$ with the solution to the normal equations of (\ref{eq:TSP}) to get
\begin{equation}\label{eq:2}
 \min_{\lambda} \|x_0 +L_k(H_{k+1,k}^{T}H_{k+1,k} + \lambda ^2 I)^{-1}H_{k+1,k}^{T} \beta e_1 - x_{\rm true}\|_{2}^{2}
\end{equation}
and use the SVD of $H_{k+1,k} = U_k \Sigma_k V_k^T$, to simplify (\ref{eq:2})
\begin{equation}\label{eq:3}
 \min_{\lambda} \|x_0 + L_k(V_k \Sigma_k ^T \Sigma_k V_k^T + \lambda^2 I)^{-1} V_k \Sigma_k ^{T} U_k^{T} \beta e_1  - x_{\rm true}\|_{2}^{2}.
\end{equation}
Thus, (\ref{eq:3}) is equivalent to 
\begin{equation}\label{eq:4}
 \min_{\lambda} \|x_{\lambda,k} - x_{\rm true}\|_2^{2} = \min_{\lambda} \| x_0 + L_k V_k( \Sigma_k ^T \Sigma_k + \lambda^2 I)^{-1}  \Sigma_k ^{T} U_k^{T} \beta e_1 - x_{\rm true}||_{2}^{2}.
\end{equation}
We again emphasize that that this is not a realistic regularization parameter choice criterion since it requires the knowledge of the true solution, but we can use it to demonstrate the performance of Hybrid LSLU.

In practice, techniques like the Generalized Cross Validation (GCV) method or the Discrepancy Principle can be implemented to estimate $\lambda_k$, at each iteration.
In this paper we focus on the GCV method, which is a predictive statistics-based approach where prior estimates of the error norm are not needed \cite{Hansen2010,golub1979generalized}. Here, we assume that the regularization parameter $\lambda_k$ should be able to predict any missing information.  Although the GCV method is typically applied for the original problem, we follow a common approach in hybrid projection methods and use the GCV function for the projected problem \eqref{eq:TSP}, with matrix $H_{{k+1},k}$ from  (\ref{eq:nhr}). The chosen regularization parameter minimizes the predictive error through the minimization of the GCV function: \begin{equation} \label{eq:gcv1}  G_{H_{k+1}, \beta e_1}(\lambda) = \frac{k \| (I - H_{k+1,k} H^{\dagger}_{\lambda}) \beta e_1 \|_2^{2}}{trace(I - H_{k+1,k} H^{\dagger}_{\lambda})^2} \end{equation}
where $H^{\dagger}_{\lambda} = (H^{T}_{k+1,k}H_{k+1,k} + \lambda^2 I)^{-1} H^{T}_{k+1,k}$. 

Using the SVD of $H_{k+1,k}$ , (\ref{eq:gcv1}) can be rewritten as: 
\begin{equation} \label{eq:gcv2}  G_{H_{k+1}, \beta e_1}(\lambda) =  \displaystyle \frac{ k \beta^2 \left(\displaystyle \sum_{i=1}^{k} \left(\frac{\lambda^2}{\sigma_i^2 + \lambda^2}[U^T_k e_1]_i\right)^2 + \left(\left[U^T_k e_1\right]_{k+1}\right)^2\right)}{\left(1+ \displaystyle \sum_{i=1}^k \frac{\lambda^2}{\sigma_i^2 + \lambda^2}\right)^2},\end{equation}
with the GCV parameter at the $k$th iteration being $\displaystyle \lambda_k = \arg \min_{\lambda}   G_{H_{k+1}, \beta e_1}(\lambda)$.

The standard GCV function may not perform well for certain types of problems. For example, in statistical nonparametric modeling, the GCV function might choose parameters that are too small and thus produce a highly oscillatory approximate solution \cite{chung2008weighted}. In our study, we find that the approximate solution is overly-smooth. To avoid this phenomenon, weighted-GCV is used, where the weighted-GCV function for the projected matrix $H_{k+1,k}$ is defined as:
 \begin{eqnarray}  G( \omega ,\lambda) &=&  \frac{k \|(I - H_{k+1,k}H_{\lambda}^{\dagger}) \beta e_1 \|_2^2} {(trace(I - \omega H_{k+1,k}H_{\lambda}^{\dagger}))^2} \label{eq:wgcv_projected1}  \\
&=& \frac{ k \beta^2 \left( \displaystyle \sum_{i=1}^{k} \left(\frac{\lambda^2}{\sigma^2_i + \lambda^2} \left[U^T_k e_{1}\right]_i \right)^2 \left(\left[U^T_k e_1 \right]_{k+1} \right)^2 \right)}{\left( 1 + \displaystyle\sum_{i=1}^k  \frac{(1-\omega)\sigma_i^2 + \lambda^2}{\sigma_i^2 + \lambda^2}\right)^2}. \label{eq:wgcv_projected2}  \end{eqnarray}  
Here, the denominator depends on a new parameter $\omega$. Similar to the selection of the regularization parameter, we find that our choice of $\omega$ impacts the smoothness of the approximate solution. Thus, we must be careful in how we select the value for $\omega$. 
If $\omega = 1$, then (\ref{eq:wgcv_projected1}) becomes the standard GCV function (\ref{eq:gcv1}). 
 If $\omega > 1$, then we are subtracting a multiple of the filter factors thus producing less smooth solutions. Likewise if $\omega < 1$, then we are adding a multiple which produces smoother solutions. Therefore, we want the value of $\omega$ to be in the following range: $0 \leq \omega \leq 1$. 

From (\ref{eq:gcv2}), it is evident that smaller regularization parameters will produce a better regularized approximate solution. As a result, we adopt the approach described in \cite{renaut2017hybrid} for selecting $\omega$. That is, we let $\omega = \frac{k+1}{m}$, where $m$ is the number of rows in the full dimension problem (\ref{eq:ip}). 

\subsubsection{Stopping Criterion} 
\label{sec:stopping}
Next we describe an approach to determine a suitable stopping criterion for Hybrid LSLU. Similar to the approach described in \cite{chung2008weighted} and inspired by \cite{bjorck1994implicit}, we assume that $\lambda$ is fixed and seek a stopping iteration $k$ that minimizes a GCV function in terms of $k$,
\begin{equation} \label{eq:stop1} \frac{n \|(I - AA_k^{\dagger})b\|_2^2}{(trace(I-AA^{\dagger}_k))^2} \approx \frac{n \| D_{k+1}^\dagger (I - A A_k^\dagger) b\|_2^2}{ ( {\rm trace} (I - A A_k^\dagger))^2} = \hat{G}(k),
\end{equation}
where $A^{\dagger}_k$ is defined by considering the approximate solution produced by Hybrid LSLU, where, without loss of generality and to simplify the notation, we have considered $x_0=0$:
$$ x_k = L_ky_{\lambda,k} = L_{k}H^{\dagger}_{\lambda} D^{\dagger}_{k+1}b \equiv A_k^{\dagger}b.$$ Since $D_{k+1}$ lacks orthonormal columns, the left-hand side of (\ref{eq:stop1}) cannot be computed exactly, as it is done in \cite{chung2008weighted}, so we use the approximation
\begin{eqnarray*} 
    n\|(I- AA^{\dagger}_k)b\|_2^2 &\approx& n \| D_{k+1}^\dagger (I - A A_k^\dagger) b\|_2^2 \\ &=& n\|(I - H_{k+1,k} H^{\dagger}_{\lambda})D^{\dagger}_{k+1}b\|_2^2
\end{eqnarray*}
where $D^{\dagger}_{k+1}AL_k = D^{\dagger}_{k+1}D_{k+1}H_{k+1,k} = H_{k+1,k}$ and $D^{\dagger}_{k+1}b = \beta e_1$. Using the SVD of $H_{k+1,k}$, the previous expression can be rewritten as:
$$n \| D_{k+1}^\dagger (I - A A_k^\dagger) b\|_2^2 = n \beta^2 \left(\displaystyle \left(\sum_{i=1}^{k} \frac{\lambda^2}{\sigma^2_i + \lambda^2} + \left[U^T_k e_1 \right]_i\right)^2 + \left(\left[ U^T_k e_1 \right]_{k+1} \right)^2 \right). 
$$
The denominator of (\ref{eq:stop1}) is equivalent to:
\begin{eqnarray*} (trace(I - AA_k^{\dagger}))^2 &=&  (trace(I - AL_k H^{\dagger}_{\lambda}D^{\dagger}_{k+1}))^2 \\ &=& (trace(I - D_{k+1}H_{k+1,k}H^{\dagger}_{\lambda}D^{\dagger}_{k+1}))^2 \\ &=& (trace(I) - trace(H_{k+1,k}H^{\dagger}_{\lambda}))^2 \\ &=& \left((m-k) + \displaystyle\sum_{i=1}^{k} \frac{\lambda_k^2}{\sigma_i^2 + \lambda_k^2}\right)^2.
 \end{eqnarray*}
Therefore the left-hand side of (\ref{eq:stop1}) can be approximated by \begin{equation} \label{eq:stop2} \hat{G}(k) = \frac{n \beta^2 \left(\displaystyle \left(\sum_{i=1}^{k} \frac{\lambda^2}{\sigma^2_i + \lambda^2} + \left[U^T_k e_1 \right]_i\right)^2 + \left(\left[ U^T_k e_1 \right]_{k+1} \right)^2 \right)}{ \left((m-k) + \displaystyle\sum_{i=1}^{k} \frac{\lambda_k^2}{\sigma_i^2 + \lambda_k^2}\right)^2}.
\end{equation}
$\hat{G}(k)$ is used to determine the stopping iteration, $k$. The algorithm will terminate when the difference between the values is small: 
\begin{equation} 
\label{eq:tol} \left|\frac{\hat{G}(k+1) - \hat{G}(k)}{\hat{G}(1)} \right|  < tol\,. 
\end{equation}

\section{Numerical Results}\label{sec:num}
We now illustrate the effectiveness of Hybrid LSLU in comparison to Hybrid LSQR \cite{paige1982lsqr}. We use three different test problems: PRtomo, PRspherical, and PRseismic from the IR tools package \cite{gazzola2018ir}. PRtomo generates data for X-ray tomographic reconstruction problems. PRspherical formulates  a tomography test problem based on the spherical Radon transform where data consists of integrals along circles. This type of problem arises in photoacoustic imaging. PRseismic creates a seismic travel-time tomography problem. These problems involve images with $256 \times 256$ pixels and correspond to a matrix $A$ that is $65160 \times 65536$ (PRtomo), $65522 \times 65536$ (PRspherical), and $131072 \times 65536$ (PRseismic) with a noise level of $$ \frac{\| e \|}{\| Ax_{true} \|} = 10^{-2}.$$
The noisy observations are provided in the top row of \Cref{fig:PR}.  

\subsection{Regularization and reconstruction performance}

We compute the reconstructed images for each problem using the proposed Hybrid LSLU method, using wGCV to select the regularization parameter and GCV for the stopping criterion.  The reconstructions are provided in the bottom row of \Cref{fig:PR}.

\begin{figure}[ht]
\centering
\begin{tabular}{ccc}
    {\scriptsize PRtomo} &  {\scriptsize PRspherical} & {\scriptsize PRseismic}\\ 
    \includegraphics[width=3.9cm]{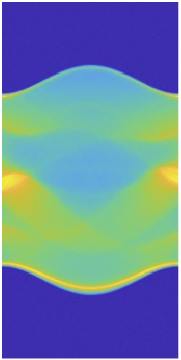} &  \includegraphics[width=3.9cm]{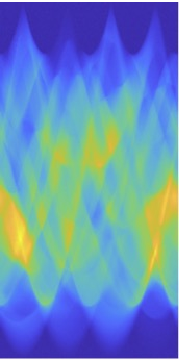} &
    \includegraphics[width=3.9cm]{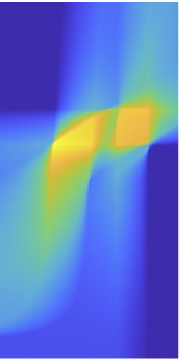} \\
\end{tabular} \\
\begin{tabular}{ccc}
    \includegraphics[width=3.9cm]{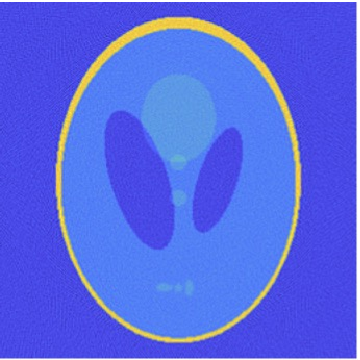} &  \includegraphics[width=3.9cm]{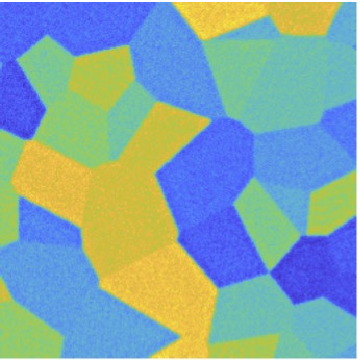} &
    \includegraphics[width=3.9cm]{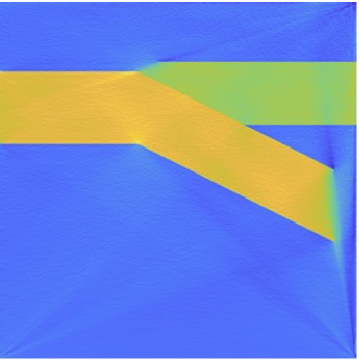} \\
\end{tabular}
\caption{Measured noisy data, $b$ (top row) and reconstructed images using Hybrid LSLU (bottom row). The image proportions are accurate but, to aid visualization, the relative size between images is not.}
\label{fig:PR}
\end{figure}

Next, in \Cref{fig:PR2} we provide the relative reconstruction error norms per iteration of Hybrid LSLU with both the wGCV and optimal regularization parameter.  Results for Hybrid LSQR with wGCV are provided for comparison. \begin{figure}[h]
\centering
\begin{tabular}{ccc}
    {\scriptsize PRtomo} &  {\scriptsize PRspherical} & {\scriptsize PRseismic} \\ 
   \includegraphics[width=3.9cm]{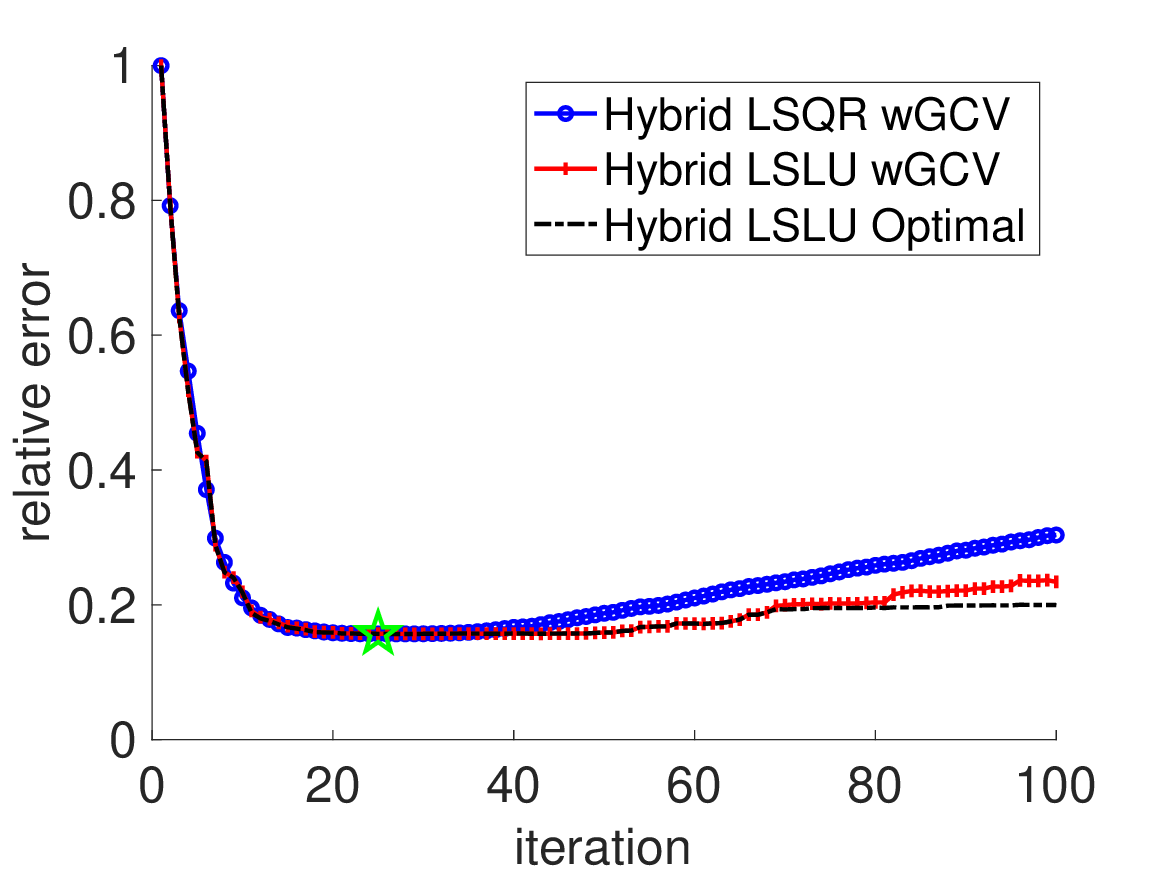} &  \includegraphics[width=3.9cm] {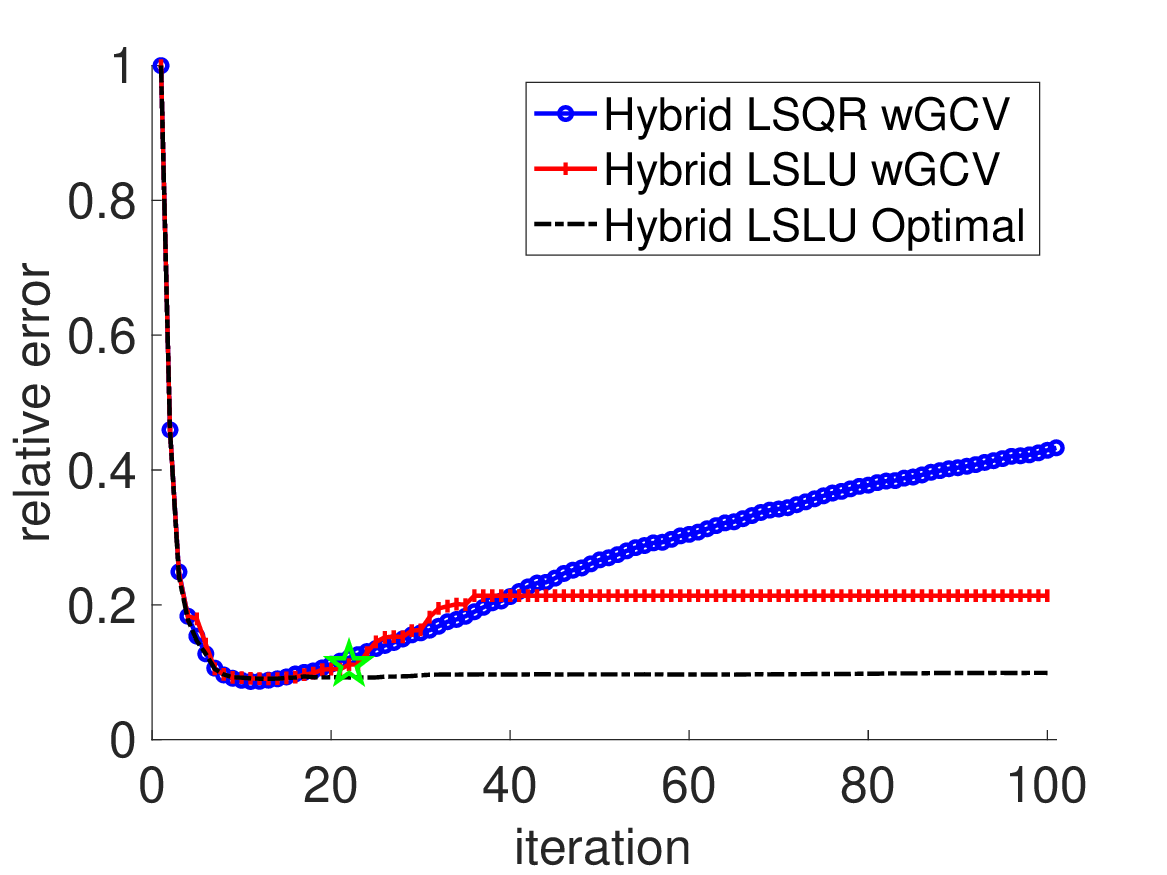} &
   \includegraphics[width=3.9cm]{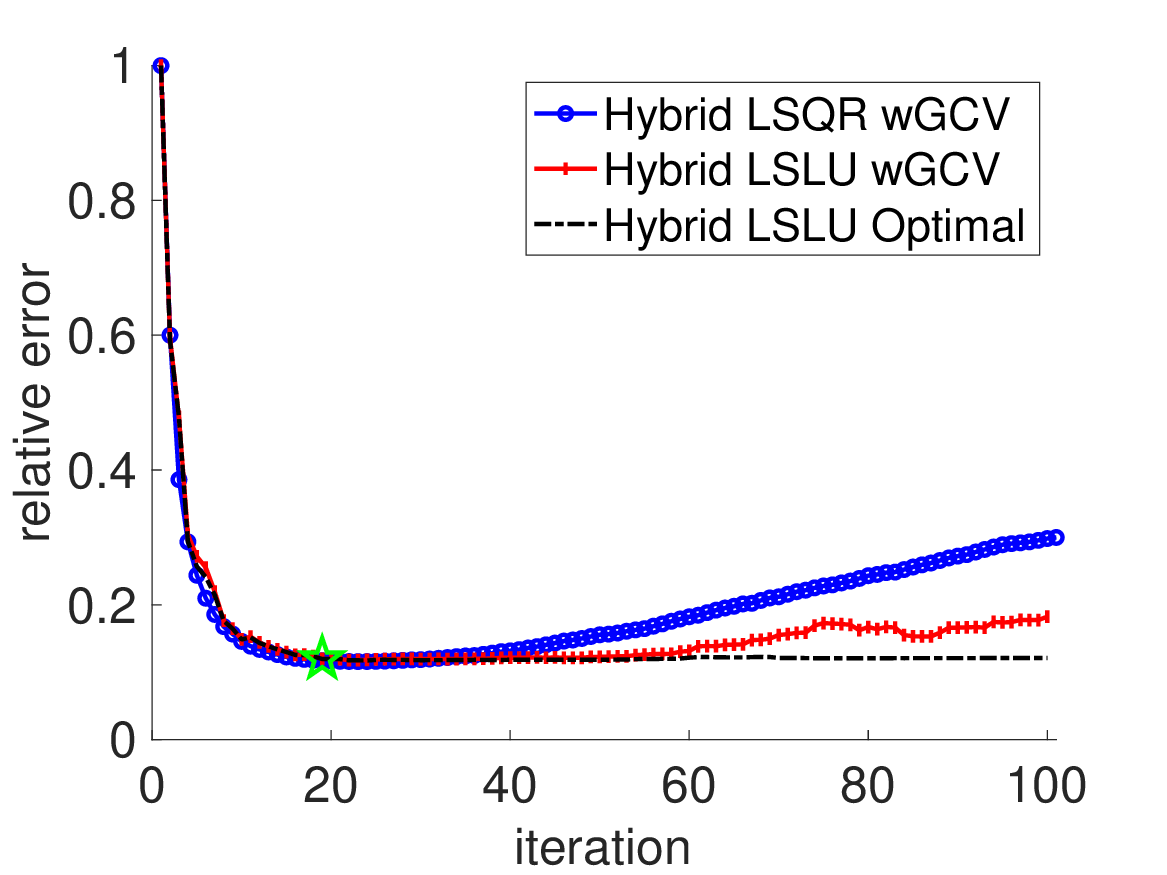} \\
    \end{tabular}
\caption{Relative reconstruction error norms per iteration of Hybrid LSLU with wGCV and the optimal regularization parameter. The automatically selected stopping iteration is highlighted with a star. Results for Hybrid LSQR with wGCV are provided for reference.}
\label{fig:PR2}
\end{figure}
From the Hybrid LSLU with optimal regularization plot, we observe that relative reconstruction error norms decrease and flatten at a nearly optimal value, which means that if a suitable choice of the regularization parameter is selected, Hybrid LSLU can provide a good regularized solution. From the Hybrid LSLU wGCV curve, we see that the relative reconstruction errors decay initially but increase slightly before leveling out.  The error is still smaller than that of Hybrid LSQR, and we remark that the stopping criteria for Hybrid LSLU performs well.
We observe that the Hybrid LSQR method exhibits faster semiconvergence, implying that the wGCV method is not a good regularization parameter choice for Hybrid LSQR in this scenario. These results demonstrate that Hybrid LSLU can provide comparable performance to Hybrid LSQR, with the same storage requirements, lower computation cost and with the benefit of being inner-product free.

\begin{center}
\begin{table}[h!]
\begin{tabular}{ |c|c|c|c|c| } 
\hline
\multicolumn{5}{|c|}{PRtomo} \\
\hline
Method & Noise Level & Stopping Iteration & Reg Parameter & Relative Error\\
\hline
\multirow{3}{4em}{Hybrid LSLU} & $10^{-3}$ & $21$ & $0.0046$ & $0.1436$ \\ 
& $10^{-2}$ & $25$ & $0.0099$ & $0.1571$ \\ 
& $10^{-1}$ & $20$ & $0.0285$ & $0.6211$ \\ 
\hline
\multirow{3}{4em}{Hybrid LSQR} & $10^{-3}$ & $50$ & $0.0051$ & $0.1285$ \\ 
& $10^{-2}$ & $99$ & $0.0109$ & $0.3035$ \\ 
& $10^{-1}$ & $100$ & $0.0105$ & $3.0472$ \\ 
\hline
\end{tabular} 
\\
\begin{tabular}{ |c|c|c|c|c| } 
\hline
\multicolumn{5}{|c|}{PRspherical} \\
\hline
Method & Noise Level & Stopping Iteration & Reg Parameter & Relative Error\\
\hline
\multirow{3}{4em}{Hybrid LSLU} & $10^{-3}$ & $31$ & $8.4043 \times 10^{-5}$ & $0.0547$ \\ 
& $10^{-2}$ & $22$ & $1.3522 \times 10^{-4}$ & $0.1112$ \\ 
& $10^{-1}$ & $22$ & $1.1716 \times 10^{-4}$ & $1.1714$ \\ 
\hline
\multirow{3}{4em}{Hybrid LSQR} & $10^{-3}$ & $43$ & $6.5967 \times 10^{-5}$ & $0.0523$ \\ 
& $10^{-2}$ & $100$ & $7.0850 \times 10^{-5}$ & $0.4329$ \\ 
& $10^{-1}$ & $100$ & $6.4423 \times 10^{-5}$ & $4.4158$ \\ 
\hline
\end{tabular}

\begin{tabular}{ |c|c|c|c|c| } 
\hline
\multicolumn{5}{|c|}{PRseismic} \\
\hline
Method & Noise Level & Stopping Iteration & Reg Parameter & Relative Error\\
\hline
\multirow{3}{4em}{Hybrid LSLU} & $10^{-3}$ & $24$ & $0.275$ & $0.1010$ \\ 
& $10^{-2}$ & $19$ & $0.0476$ & $0.1198$ \\ 
& $10^{-1}$ & $48$ & $0.0379$ & $0.8514$ \\ 
\hline
\multirow{3}{4em}{Hybrid LSQR} & $10^{-3}$ & $50$ & $0.0113$ & $0.0875$\\ 
& $10^{-2}$ & $100$ & $0.0277$ & $0.2999$ \\ 
& $10^{-1}$ & $100$ & $0.0259$ & $3.1474$ \\ 
\hline
\end{tabular}

\caption{Numerical results for the three test problems PRtomo, PRspherical, and PRseismic for various noise levels.  Regularization parameters and relative errors correspond to values at the stopping iteration.}
\label{tab:Table1}
\end{table}
\end{center}

The performance of Hybrid LSLU and Hybrid LSQR is similar for various noise levels.  In \Cref{tab:Table1}, we provide the automatically selected stopping iteration, the computed regularization parameter using wGCV, and the relative reconstruction error norm, for noise levels $10^{-3}$, $10^{-2}$, and $10^{-1}$.  We remark that the results for $10^{-2}$ are consistent with the results in \Cref{fig:PR2}. We observe that for lower noise levels, Hybrid LSLU and Hybrid LSQR perform comparably, but as the noise level increases,  
Hybrid LSLU appears to perform better for all three test problems.  This may be attributed to the stopping criteria and selected regularization parameter that result in better reconstructions for Hybrid LSLU.

\begin{figure}[h]
\centering
\includegraphics[width=\textwidth]{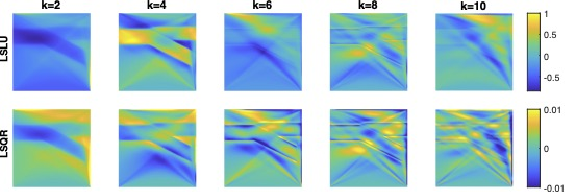}
\caption{Basis vectors for the Krylov subspace (\ref{eq:K1}) generated by LSLU and LSQR at iterations $k=2,4,6,8,10$ for the PRseismic example.}
\label{fig:PR3}
\end{figure}

\begin{figure}[t]
\centering
\includegraphics[width=\textwidth]{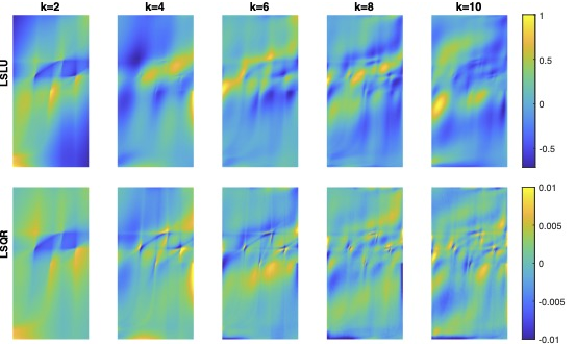}
\caption{Basis vectors for the Krylov subspace (\ref{eq:K2}) generated by LSLU and LSQR at iterations $k=2,4,6,8,10$ for the PRseismic example.}
\label{fig:PR4}
\end{figure}

Finally, for the PRseismic example, we investigate the images created by the basis vectors in \Cref{fig:PR3,fig:PR4}.  Recall that in Hybrid LSLU, two sets of basis vectors are being constructed in an inner-product free manner, one for each of the Krylov subspaces \eqref{eq:K1} and \eqref{eq:K2}. We display $5$ of the columns of the basis vectors from $L_k$ and $D_k$ reshaped into images of corresponding sizes for Hybrid LSLU, and provide the corresponding columns of the basis vectors generated via Hybrid LSQR for comparison.  We observe that although the vectors generated via Hybrid LSLU
at the $k$th iteration span the same subspace as the vectors generated via Hybrid LSQR, they have different features. We observe that the columns of $L_k$ retrieve some characteristics of the true solution in early iterations; hence we expect this to be a good basis for the solution. The columns of $D_k$ contain information regarding the measurement $b \in \mathbb{R}^m$ or residual space, where the columns are basis vectors for ${\cal{K}}_k(AA^T,r_0)$. The ability of $L_k$ to contain parts of the regularized solution is beneficial in helping to produce an accurate approximation of the true solution. 

These images help to understand how different methods pick up different information. LSQR basis picks up the high-frequency information, due to the orthogonality requirement, and has a nice connection to SVD (frequency analysis). LSLU seems to picks up high-frequency information but doesn't project out previous vectors.

\subsection{Low-rank approximation for uncertainty quantification}
A natural question is how the low-rank approximation resulting from the inner-product free Hessenberg process can be exploited for efficient uncertainty quantification.  We follow similar works that use low-rank perturbative approximations for the posterior covariance matrix, see e.g., ~\cite{flath2011fast,bui2012extreme, bui2013computational, spantini2015optimal,saibaba2015fastc}. 
We focus on the simple Gaussian, linear case where the model is given by \eqref{eq:ip} with noise $e \sim \mathcal{N}(0, \sigma^2 I)$, prior $x \sim \mathcal{N}(0, \alpha^2 I)$, and posterior $x \mid b \sim \mathcal{N}(x_{\rm MAP}, \Gamma)$ with $x_{\rm MAP}$ the solution to optimization problem \eqref{eq:vr} with $\lambda = \frac{\sigma}{\alpha}$ and $\Gamma = (\frac{1}{\alpha^2}I + \frac{1}{\sigma^2} A^T A)^{-1}$. Assume that a hybrid projection method such as Hybrid LSLU was used to compute an estimate $x_{\lambda_k,k}$ of $x_{\rm MAP}$ where $\lambda_k$ was determined using techniques in Section \ref{sec:paramselect}. After $k$ iterations of the LSLU projection process, we have matrices $L_k$ and $D_{k}$ with linearly independent columns, $H_{k+1,k}$, and $W_{k}$ that satisfy \eqref{eq:nhr} and \eqref{eq:nhr2}. Consider the following low rank approximation,
\begin{align}
A^T A & \approx A^T D_{k} D_{k}^\dagger A \\
& = L_{k} W_{k} (D_{k}^T D_{k})^{-1} W_{k}^T L_{k}^T \,.
\end{align}
Let $G_{k}^T G_{k}  = W_{k} (D_{k}^T D_{k})^{-1} W_{k}^T = V^\text{\tiny G} (\Sigma^\text{\tiny G})^T \Sigma^\text{\tiny G} (V^\text{\tiny G})^T$ be its eigenvalue decomposition with
eigenvalues $(\sigma_1^\text{\tiny G})^2, \ldots, (\sigma_k^\text{\tiny G})^2$
and let $Z_k = L_{k} V^\text{\tiny G}$, then we get the following low-rank approximation,
 \begin{equation}
   \label{eq:GKapprox}
   A^T A \approx \> L_k G_{k}^T G_{k} L_k^\top = \> Z_k (\Sigma^\text{\tiny G})^T \Sigma^\text{\tiny G} Z_k^T\,.
 \end{equation}

In the following, we investigate the use of low-rank approximation \eqref{eq:GKapprox} for estimating solution variances and sum of variances and covariances, as was done in \cite{chung2024computational}. The posterior variances provide a measure of the spread of the posterior distribution around the posterior mean and correspond to the diagonal elements of the posterior covariance matrix $\Gamma$.  For many problems, $\Gamma$ is large and dense, so forming it explicitly to obtain the diagonal entries may be infeasible.

Assume that we have an estimate of the noise variance $\sigma^2$ and fix $\lambda$.  Here we use the true noise variance and the regularization parameter from Hybrid LSQR wGCV. Then using~\cref{eq:GKapprox} and the Woodbury formula, we obtain the approximation
 \begin{align*}
 \Gamma = \sigma^2 (\lambda I_n +  A^T A)^{-1} \approx & \>  \sigma^2 (\lambda I_n + Z_k (\Sigma^\text{\tiny G})^T \Sigma^\text{\tiny G} Z_k^T)^{-1} \\
  = &  \>    \sigma^2 ({\lambda^{-1}} I_n - \lambda^{-1}Z_k (Z_k^T Z_k  + \lambda ((\Sigma^\text{\tiny G})^T \Sigma^\text{\tiny G} )^{-1})^{-1} Z_k^T )\\\
  = &  \>   \sigma^2 ({\lambda^{-1}} I_n - Z_k\Delta_k Z_k^T) =: \Gamma_k
 \end{align*}
 where
 \begin{equation}
 \Delta_k \equiv (Z_k^T Z_k  + \lambda ((\Sigma^\text{\tiny G})^T \Sigma^\text{\tiny G} )^{-1})^{-1}.
\end{equation}
Notice that contrary to previous approaches, $Z_k$ does not contain orthonormal columns and thus $\Delta_k$ is not a diagonal matrix.  However, it is a $k \times k$ matrix, so for reasonably sized $k$, this computation is not a burden.

Notice that we have an efficient representation of $\Gamma_k$ as a low-rank perturbation of the prior covariance matrix, $\sigma^2\lambda^{-1} I_n$.
Thus, diagonal entries of $\Gamma_k$ can provide estimates of diagonal entries of $\Gamma$, where the main computational requirement is to obtain the diagonals of the rank-$k$ perturbation. In addition, one can approximate the sum of all values in the posterior covariance matrix $1^T \Gamma\, 1$ as
\begin{equation}
\label{eq:UQsum}
     1^T \Gamma\, 1 =  \sigma^2 ({\lambda^{-1}} n - 1^T Z_k \Delta_k Z_k^T 1)
\end{equation}
 where $1$ is an $n \times 1$ vector of ones.

In the left plot of \cref{fig:UQ}, we provide estimates of \cref{eq:UQsum} at various iterations $k$ of the LSQR process and the LSLU process. We observe that the LSLU approximations of the sum of elements in the posterior covariance matrix are nearly indistinguishable from the LSQR approximations.  The absolute difference per iteration is provided for reference. For the automatically selected stopping iteration, denoted with a vertical line, we provide an image of the solution variances (corresponding to the diagonal entries of $\Gamma_k$) for both LSQR and LSLU. 

\begin{figure}[bt]
\begin{center}
\begin{tabular}{cc}
\includegraphics[width=6cm]{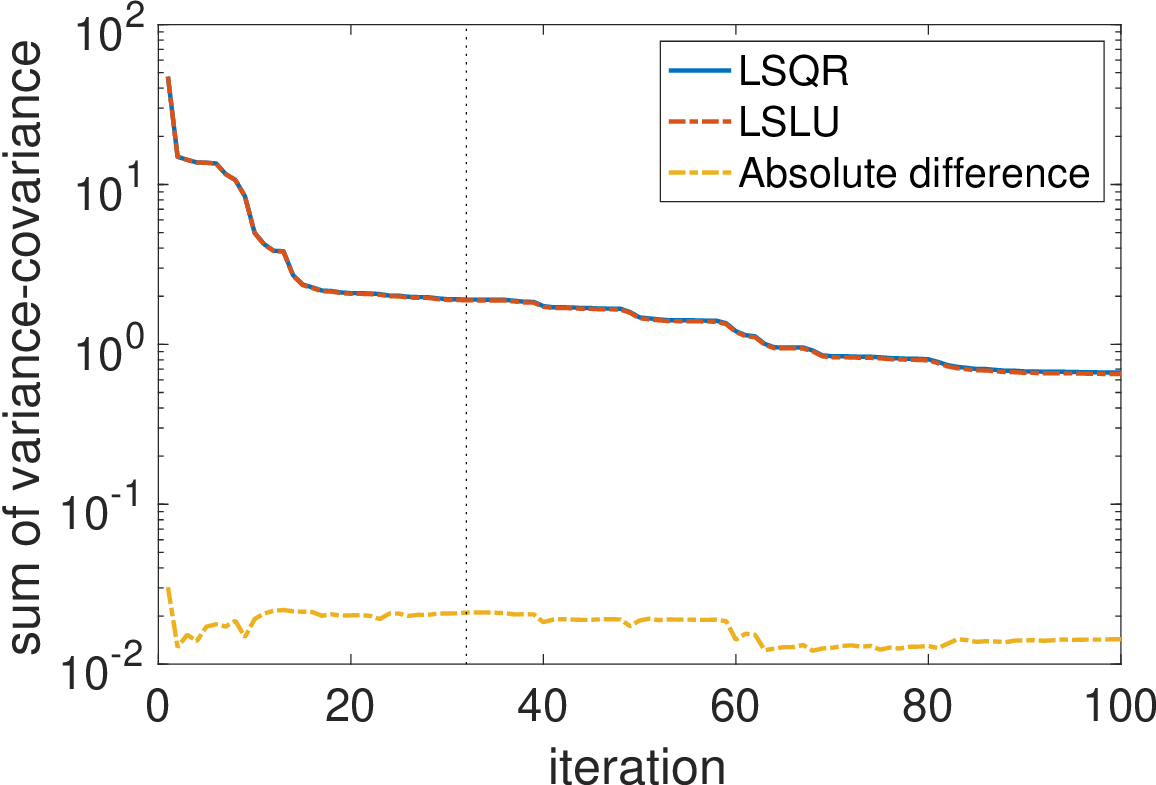} &
\includegraphics[width=5.9cm]{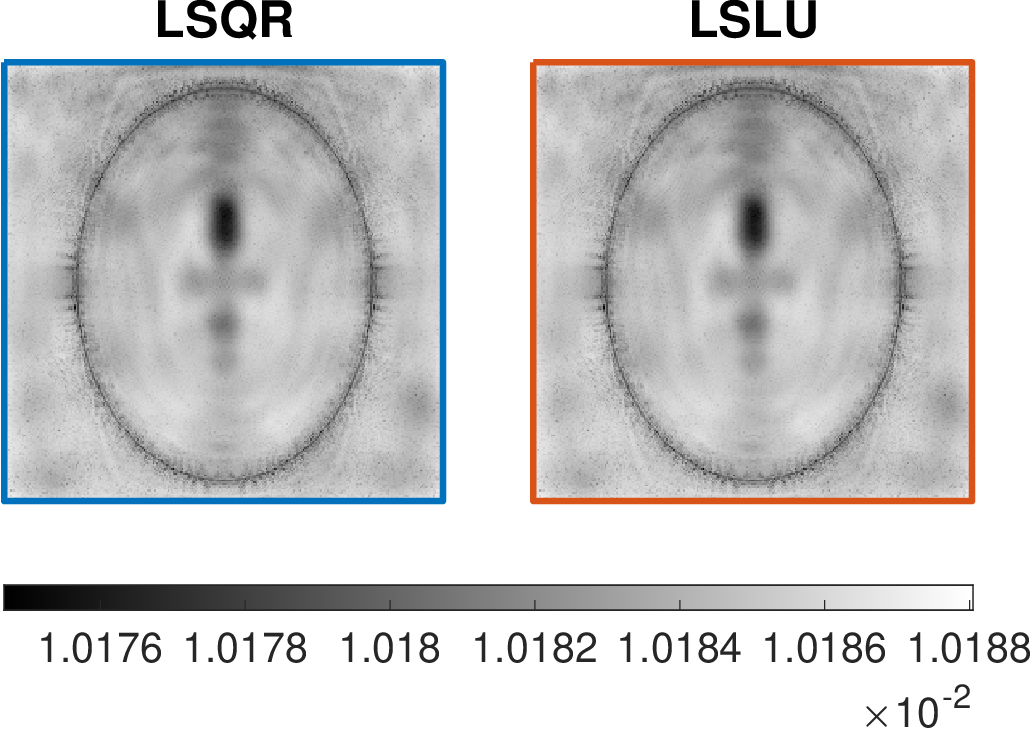} \end{tabular}
\caption{For the PRtomo example, we provide estimates of $1^T \Gamma_k\, 1$ per iteration $k$
using both the Golub-Kahan bidiagonalization approximation (denoted LSQR) and the inner-product free Hessenberg approximation (denoted LSLU) approximation.  Solution variances computed at the stopping iterate (corresponding to the vertical line) are provided for LSQR and LSLU on the right. }
\label{fig:UQ}
\end{center}
\end{figure}

The results for uncertainty quantification estimation using the low-rank approximations from the Hessenberg process (Hybrid LSLU) and the Golub-Kahan bidiagonalization process (Hybrid LSQR) are  very close. 
Both approaches have similar storage costs, but the Hessenberg process has the added benefits of avoiding inner products and avoiding reorthogonalization costs. 
The tradeoff is that since the basis vectors are no longer orthonormal, $\Delta_k$ is no longer diagonal, and we must work with a $k \times k$ matrix.

\section{Conclusions}\label{sec:conc}
In this paper, we introduced two new inner-product free Krylov methods for rectangular large-scale linear ill-posed inverse problems. Based on our numerical observations, the Hybrid LSLU
method is comparable to Hybrid LSQR in its ability to select regularization parameters during the iterative process and stabilize semiconvergence.  Both approaches only require matrix-vector multiplications with $A$ and its transpose, making them appealing for large-scale problems. Hybrid LSLU has the added benefit of being inner-product free which could be useful in solving problems with mixed-precision and parallel computing.

\section*{Acknowledgments}
This work was partially funded by the U.S. National Science Foundation, under grants DMS-2038118\@, DMS-2341843 and DMS-2208294. Any opinions, finding, and conclusions or recommendations expressed in this material are those of the author(s) and do not necessarily reflect the views of the National Science Foundation. 

\bibliographystyle{siamplain}
\bibliography{references}

\end{document}